\documentclass[11pt,a4paper]{article}

\usepackage{authblk}
\usepackage{tabularx,booktabs}
\usepackage{fullpage}
\usepackage{graphicx}
\usepackage{lipsum}
\usepackage[utf8]{inputenc}
\usepackage[margin=1.0in]{geometry}
\usepackage[hidelinks]{hyperref}
\usepackage[T1]{fontenc}
\usepackage[utf8]{inputenc}
\usepackage{amsmath,amssymb, amsthm}
\usepackage{bm}
\usepackage{algorithm}
\usepackage{algorithmic}
\usepackage{placeins}
\usepackage{dsfont}
\usepackage{url}
\usepackage{svg}
\usepackage{graphics}
\usepackage[export]{adjustbox}
\usepackage{todonotes}
\usepackage{soul}
\usepackage{bbding}
\usepackage{dsfont}
\usepackage[labelfont=bf]{caption}

\usepackage{algorithm}
\usepackage{algorithmic}
\usepackage{amsmath,amssymb}

\newcommand{\junk}[1]{}

\newtheorem{theorem}{Theorem}[section]
\newtheorem{corollary}{Corollary}[theorem]
\newtheorem{lemma}[theorem]{Lemma}
\newtheorem{proposition}[theorem]{Proposition}

\begin{document}
\thispagestyle{empty}

\title{Higher Accuracy Modular Data Assimilation for the Navier-Stokes Equations}

\author[]{
Troy Yang \footnote{Department of Mathematics, University of Pittsburgh, Pittsburgh, PA, 15260, USA; Email: try21@pitt.edu. Partially supported by NSF grant DMS 2410893.
}}

\date{}

\maketitle

\begin{abstract}
  This paper develops an accurate and effective combination of second order backward differentiation time discretization (BDF2) with modular, 2-step nudging-based data assimilation
\begin{align}
    \text{Forecast step: } &\frac{3\Tilde{v}^{n+2}-4v^{n+1}+v^n}{2\Delta t}+\Tilde{v}^{n+2} \cdot \nabla \Tilde{v}^{n+2} - \nu \Delta \Tilde{v}^{n+2} + \nabla q^{n+2}=f(x) \notag \\ &\nabla \cdot \Tilde{v}^{n+2} = 0 \notag \\
    \text{Analysis step: }
    &\frac{3v^{n+2}-3\Tilde{v}^{n+2}}{2\Delta t}-\chi I_H(u(t^{n+2})-v^{n+2})=0. \notag 
\end{align}
If $I_H=I_H^2$, the analysis step can be made explicit, taking the form
\begin{align}
    v^{n+2}=\Tilde{v}^{n+2}+\frac{2\Delta t\chi}{3+2\Delta t\chi}I_H(u^{n+2}-\Tilde{v}^{n+2}). \notag
\end{align}
This implies the analysis step has the stability property of an implicit step and lower complexity than an explicit analysis step. Stability and error estimates for the BDF2 scheme are presented along with their proofs. Numerical experiments are conducted to assess the performance of BDF2 modular assimilation algorithm. The results of the experiments support the conclusion that modular data assimilation has comparable accuracy to standard, fully coupled data assimilation while greatly reducing computational complexity and cost.
\end{abstract}

\section{Introduction}
  Data assimilation incorporates observations to produce more accurate solutions and extends the predictability horizon of high Reynolds flows \cite{Kalnay:2003}. Modular, 2-step nudging, inspired by the efficiency of Kalman filter algorithms, is a newly developed data assimilation method that splits standard nudging into a forecast step and an analysis step. Modular nudging reduces algorithmic complexity of standard nudging \cite{Fang:2025}. Previous work on modular nudging only considered the simple, over-diffused and inaccurate fully implicit method. Herein, because time accuracy is important, we extend and test the algorithm and analysis for the second order, BDF2 time discretization. Let $u$ denote the fluid velocity and $p$ the fluid pressure. We consider the internal $2d$ or $3d$ flow of an  incompressible viscous fluid in a domain $\Omega$ given by equations:

\numberwithin{equation}{section}

\begin{equation}
    u_t+u\cdot \nabla u -\nu \Delta u + \nabla p = f(x), \text{ and } \nabla \cdot u=0, \text{ in } \Omega, 0<t\leq T,
\end{equation}

\begin{equation}
    u=0 \text{ on } \partial \Omega \text{ and } u(x,0)=u_0(x).
\end{equation}

Here $H$ denotes the observation spacing, $I_H$ an interpolation operator, and $u_{obs}(x,t)=I_H u(x,t)$ observations of the flow. In this paper, we primarily consider the simplifying case when $I_H$ is $L^2$ projection. We let $v$ denote the approximate velocity. The boundary and initial conditions are $v=0$ on $\partial \Omega$ and $v(x,0)=v_0(x)$ with $v_0(x)\neq u_0(x)$. To present the ideas, the space discretization is temporarily suppressed. Pick $\chi>0$ and proceed by the forecast and analysis steps:
\begin{align}
    \text{Forecast step: } &\frac{3\Tilde{v}^{n+2}-4v^{n+1}+v^n}{2\Delta t}+\Tilde{v}^{n+2} \cdot \nabla \Tilde{v}^{n+2} - \nu \Delta \Tilde{v}^{n+2} + \nabla q^{n+2}=f(x) \\ &\nabla \cdot \Tilde{v}^{n+2} = 0 \\
    \text{Analysis step: }
    &\frac{3v^{n+2}-3\Tilde{v}^{n+2}}{2\Delta t}-\chi I_H(u(t^{n+2})-v^{n+2})=0.
\end{align}
In the forecast step, we apply BDF2 time discretization and compute an approxmiate solution $\Tilde{v}^{n+2}$. Given $\Tilde{v}^{n+2}$, we apply corrections to the computed solution by assimilating observations of the flow in the analysis step. By splitting standard nudging into two steps, we overcome algorithmic limitations and lower computational complexity \cite{Fang:2025}.

This paper is organized as follows. General notations, preliminaries, and lemmas are introduced in Section 2. Section 3 discusses how the analysis step of modular nudging can be written explicitly to reduce algorithmic complexity. In particular, we have that if $I_H=I_H^2$, the explicit analysis step takes the form
\begin{align}
    v^{n+2}=\Tilde{v}^{n+2}+\frac{2\Delta t\chi}{3+2\Delta t\chi}I_H(u^{n+2}-\Tilde{v}^{n+2}). \notag
\end{align}
In Section 4, a BDF2 modular, 2-step nudging data assimilation algorithm is presented along with its stability analysis. We prove both finite-time and long-time stability estimates. In Section 5, we present two error estimates for the BDF2 scheme and discuss differences in their proofs. For both estimates, the errors do not grow over time. Section 6 presents three numerical experiments. The numerical tests include cases where $I_H=I_H^2$, and the explicit formula for the analysis step is implemented. 
The first experiment confirms that the  temporal convergence rate for BDF2 modular, 2-step nudging is second order by comparing computed solution with analytical solution on a unit square domain. We also verify that modular, 2-step nudging is computationally cheaper than standard nudging by comparing CPU runtime for various time steps. The second experiment tests the accuracy of the solution for a 2d flow between offset cylinders at a higher Reynolds number without exact solution. The computed solution is compared with the referenced solution approximated using Direct Numerical Simulation (DNS). The third experiment is another accuracy test for the 2d pipe cavity flow problem. We conduct the numerical simulation using same inflow and outflow boundary conditions as well as the "do-nothing" outflow boundary condition. Conclusions and future directions follow in Section 7. Appendix A considers an alternate analysis step and makes remarks about the difficulties of proving operator commutativity encountered in the analysis.

\subsection{Related work}
Nudging was developed by Luenberger \cite{Luenberger:1964} in 1964 to produce more accurate solutions by incorporating noisy observational data. Standard nudging has been extensively studied by Azouani, Olson, and Titi \cite{Azouani:2014} (a landmark result of the area), Biswas and Price \cite{Biswas:2021}, Cao, Giorgini, Jolly, and Pakzad \cite{Cao:2022}, Larios, Rebholz, and Zerfas \cite{Larios:2019}. Its feasibility has been proven and tested for 2D and 3D applications in those papers (and many other papers), e.g., Hoke and Anthes \cite{Hoke:1976}, Lakshmivarahan and Lewis \cite{Lakshmivarahan:2013}, Kalnay \cite{Kalnay:2003}, Navon \cite{Navon:1998}, Stauffer and Bao \cite{Stauffer:1993}, and Zou, Navon and LeDimet \cite{Zou:1992}. To overcome the algorithmic difficulties of standard nudging and inspired by the computational efficiency of Kalman filter variants, modular, 2-step nudging with backward Euler time discretization scheme was recently developed by Cibik, Fang, and Layton \cite{Fang:2025}.

\section{Preliminaries}
 \textbf{Key assumption}: We assume $X^H\subset X:=(H_0^1(\Omega))^d$, $d=2,3$, is a finite element velocity space with mesh size $H$ such that
\begin{align}
    ||(I-I_H)w||\leq CH||\nabla w||, \forall w \in (H_0^1(\Omega))^d, d=2,3.
\end{align}
We denote the pressure space as $Q:=L_0^2(\Omega)$. Let $X_h\subset X, Q_h\subset Q$ be an inf-sup stable pair of discrete velocity-pressure spaces. We use the standard $L^2(\Omega)$ notation: $||\cdot||$ for norm and $(\cdot,\cdot)$ for inner product. The dual norm of $X$ is denoted by $||\cdot ||_{-1}$. Let $L=\text{diam}(\Omega)$. In X, we have the Poincar\'{e} inequality: there exists a constant $C_P$ the depends on the shape of $\Omega$ but not $L$ such that for any $u\in X$,
\begin{align}
    ||u||\leq C_PL ||\nabla u||.
\end{align}
For smooth functions $u,v,w$, we denote the trilinear form $b:X\times X\times X \rightarrow \mathbb{R}$ by
\begin{align}
    b(u,v,w)=\frac{1}{2}(u\cdot \nabla v, w)-\frac{1}{2}(u\cdot \nabla w, v).
\end{align}
We define the global velocity by $$U=\text{ess sup}_{\Omega \times (0,\infty)}||u(x,t)||.$$ The Taylor length scale is $$l=\Big(\text{ess sup}_{\Omega \times (0,\infty)}||\nabla u(x,t)||\Big)^{-1}U.$$

The lemmas presented below are useful for the stability and error analysis. Lemma 2.1 and lemma 2.2 are cited from Rong and Fiordilino on pages 4 and 5 respectively \cite{Rong:2018}.
\begin{lemma}
    For real numbers $a,b,c\in \mathbb{R}$, we have 
    \begin{align}
        2(3a-4b+c)a=a^2-b^2+(2a-b)^2-(2b-c)^2+(a-2b+c)^2.
    \end{align}
\end{lemma}

\begin{lemma}
    If $g_t, g_{tt}, g_{ttt} \in L^2(0, T; H^r(\Omega))$, then we have
    \begin{align}
        \Bigg|\Bigg|\frac{3g^{n+1}-4g^n+g^{n-1}}{2\Delta t}-g_t(t^{n+1})\Bigg|\Bigg|_r^2\leq C \Delta t^3 \int_{t^{n-1}}^{t^{n+1}}||g_{ttt}||_r^2 dt.
    \end{align}
\end{lemma}

We also use the polarization identity extensively in the analysis.
\begin{lemma}{(Polarization identity)}
    For u, v $\in X$, we have 
    \begin{align}
        (u,v)=\frac{1}{2}(||u||^2+||v||^2-||u-v||^2).
    \end{align}
\end{lemma}

\section{The analysis step written explicitly}
  
With $I_H$ a projection, running the analysis step in finite elements requires spatial discretization on a mesh of width $h\ll H$ for the equation
\begin{align}
    \frac{3v^{n+2}-3\Tilde{v}^{n+2}}{2\Delta t}-\chi I_H(u(t^{n+2})-v^{n+2})=0
\end{align}
and solving the corresponding linear system for the discrete approximation to $v^{n+2}$. We show, for the modified analysis step herein, that the analysis step can be written explicitly and requires only one explicit application of the operator $I_H$ provided that $I_H=I_H^2$.

When $I_H=I_H^2$, then $v^{n+2}$ can be computed using an explicit step
\begin{align}
    v^{n+2}=\Tilde{v}^{n+2}+\frac{2\Delta t\chi}{3+2\Delta t\chi}I_H(u^{n+2}-\Tilde{v}^{n+2}).
\end{align}
Cibik, Fang, and Layton \cite{Fang:2025} have shown that this explicit calculation reveals the implication that modular nudging with backward Euler time discretization has a lower complexity than an explicit Kalman filter step. We extend this result to BDF2 time discretization.

\begin{proposition}
    Assume $3I+2\Delta t \chi I_H$ is invertible, the analysis step satisfies
    \begin{align}
        v^{n+2}=\Tilde{v}^{n+2}+\frac{2\Delta t\chi}{3+2\Delta t\chi}I_H(u^{n+2}-\Tilde{v}^{n+2})+\frac{(2\Delta t\chi)^2}{3(3+2\Delta t\chi)}[I_H-I_H^2](u^{n+2}-v^{n+2}).
    \end{align}
\end{proposition}

\begin{proof}
    One can insert $v^{n+2}$ in (3.3) into (3.1) and immediately verify it satisfies (3.1). To explain the source of (3.3), we derive (3.3) as follows. We start with 
    \begin{align}
        3v^{n+2}+2\Delta t\chi I_Hv^{n+2}=3\Tilde{v}^{n+2}+2\Delta t\chi I_Hu^{n+2}.
    \end{align}
    We add and subtract $2\Delta t\chi I_H \Tilde{v}^{n+2}$ to RHS:
    \begin{align}
        (3I+2\Delta t\chi I_H)v^{n+2}=(3I+2\Delta t\chi I_H)\Tilde{v}^{n+2}+2\Delta t\chi I_H(u^{n+2}-\Tilde{v}^{n+2}).
    \end{align}
    We apply the operator $(3I+2\Delta t\chi I_H)^{-1}$ to both sides:
    \begin{align}
        v^{n+2}=\Tilde{v}^{n+2}+2\Delta t\chi(3I+2\Delta t\chi I_H)^{-1}I_H(u^{n+2}-\Tilde{v}^{n+2}).
    \end{align}
    This proves (3.2). When $(3I+2\Delta t\chi I_H)^{-1}$ exists, the inversion identity holds:
    \begin{align}
        \Bigg(3I+2\Delta t\chi I_H\Bigg)^{-1}I_H=\frac{1}{3+2\Delta t\chi}I_H+\frac{2\Delta t\chi}{3+2\Delta t\chi}(I_H-I_H^2).
    \end{align}
    Substituting identity into (3.6) and simplifying terms completes the proof of (3.3).
\end{proof}

\begin{corollary}
    If $I_H=I_H^2$, then
    \begin{align}
        v^{n+2}=\Tilde{v}^{n+2}+\frac{2\Delta t\chi}{3+2\Delta t\chi}I_H(u^{n+2}-\Tilde{v}^{n+2}).
    \end{align}
\end{corollary}

\begin{proof}
    Take (3.3) and use the assumption that $I_H=I_H^2$.
\end{proof}
   
\section{Stability of BDF2 modular nudging}
  
Extending the algorithms in \cite{Fang:2025}, we propose the following modular, 2-step nudging-based data assimilation algorithm for approximating the solutions of (1.1).

\begin{algorithm}[H]
\caption{}\label{alg:cap1}
\begin{algorithmic}
\STATE \begin{align}
    \text{Forecast step: } 
    &\text{Given } v^n, v^{n+1} \in X_h, \text{ find } (\Tilde{v}^{n+2}, q^{n+2}) \in (X_h, Q_h) \text{ satisfying:} \notag \\ 
    &\Bigg(\frac{3\Tilde{v}^{n+2}-4v^{n+1}+v^n}{2\Delta t}, v_h\Bigg)+b(\Tilde{v}^{n+2}, \Tilde{v}^{n+2}, v_h) + \nu (\nabla \Tilde{v}^{n+2}, \nabla v_h) \\
    &- (q^{n+2}, \nabla \cdot v_h)=(f^{n+2},v_h) \text{ } \forall v_h \in X_h,  \notag \\ 
    &(\nabla \cdot \Tilde{v}^{n+2}, q_h) = 0 \text{ } \forall q_h \in Q_h.
\end{align}

\IF {no data at $t^{n+2}$}
    \STATE \begin{align}
        \text{proceed to next forecast step} \notag 
    \end{align}
\ELSE 
\STATE \begin{align}
    \text{Analysis step: }
    &\Bigg(\frac{3v^{n+2}-3\Tilde{v}^{n+2}}{2\Delta t}, v_h\Bigg) -\chi (I_H(u(t^{n+2})-v^{n+2}), v_h)=0 \text{ } \forall v_h \in X_h.
\end{align}
\ENDIF

\end{algorithmic}
\end{algorithm}

In this stability analysis, we focus on fully nonlinear convection term. Analysis with linearly implicit convection terms remains an open problem. Before presenting the stability estimate, we begin by proving a lemma essential for the stability analysis. As inspired by results developed by Rong and Fiordilino on page 4 of \cite{Rong:2018}, Lemma 4.1 shows two identities derived from algebraic manipulation of the analysis step in a similar manner.

\begin{lemma}
    Consider the analysis step in (4.3). Then the following identities hold:
    \begin{align}
        &||\Tilde{v}^{n+2}||^2=||v^{n+2}||^2+||\Tilde{v}^{n+2}-v^{n+2}||^2-\frac{4}{3}\Delta t\chi (I_H(u(t^{n+2})-v^{n+2}),v^{n+2}) \\
        & and \notag \\
        &\Bigg(\frac{3v^{n+2}-4v^{n+1}+v^n}{2\Delta t}, \Tilde{v}^{n+2}-v^{n+2}\Bigg)=-\frac{1}{3}\chi (I_H(u(t^{n+2})-v^{n+2}),3v^{n+2}-4v^{n+1}+v^n).
    \end{align}
\end{lemma}

\begin{proof}
    Setting $v_h=\frac{4}{3}\Delta tv^{n+2}$ in (4.3), we have
    \begin{align}
        2||v^{n+2}||^2-2(\Tilde{v}^{n+2}, v^{n+2})-\frac{4}{3}\Delta t\chi (I_H(u(t^{n+2})-v^{n+2}),v^{n+2})=0.
    \end{align}
    We apply the polarization identity to the second term. Then
    \begin{align}
        &2||v^{n+2}||^2-2\Bigg(\frac{1}{2}||\Tilde{v}^{n+2}||^2+\frac{1}{2}||v^{n+2}||^2-\frac{1}{2}||\Tilde{v}^{n+2}-v^{n+2}||^2\Bigg) \notag \\
        &-\frac{4}{3}\Delta t\chi (I_H(u(t^{n+2})-v^{n+2}),v^{n+2})=0.
    \end{align}
    We have proven the first identity (4.4). For the second identity, we set $v_h=\frac{3v^{n+2}-4v^{n+1}+v^n}{3}$ in (4.3) and get
    \begin{align}
        \Bigg(\frac{3v^{n+2}-3\Tilde{v}^{n+2}}{2\Delta t}, \frac{3v^{n+2}-4v^{n+1}+v^n}{3}\Bigg) -\chi \Bigg(I_H(u(t^{n+2})-v^{n+2}), \frac{3v^{n+2}-4v^{n+1}+v^n}{3}\Bigg)=0.
    \end{align}
    Rearranging terms, we have
    \begin{align}
        \Bigg(\frac{3v^{n+2}-4v^{n+1}+v^n}{2\Delta t}, v^{n+2}-\Tilde{v}^{n+2}\Bigg) =\frac{1}{3}\chi \Bigg(I_H(u(t^{n+2})-v^{n+2}), 3v^{n+2}-4v^{n+1}+v^n\Bigg)
    \end{align}
    Multiplying both sides by $-1$ completes the proof.
\end{proof}

We prove stability in Theorem 4.2 next.

\begin{theorem}
    Consider the forecast step in (4.1) and the analysis step in (4.3). Assume $I_H$ is an $L^2$ projection. Then the following holds:
    \begin{align}
        &||v^N||^2+||2v^N-v^{N-1}||^2+\frac{2}{3}\Delta t\chi ||I_H(v^N-u(t^N))||^2 \notag \\
        &+\frac{2}{3}\Delta t\chi ||2I_H(v^N-u(t^N))-I_H(v^{N-1}-u(t^{N-1}))||^2 + \sum_{n=0}^{N-2}||v^{n+2}-2v^{n+1}+v^n||^2 \notag \\
        &+6\sum_{n=0}^{N-2}||\Tilde{v}^{n+2}-v^{n+2}||^2 + 3\Delta t\nu \sum_{n=0}^{N-2}||\nabla \Tilde{v}^{n+2}||^2 +\frac{4}{3}\Delta t\chi \sum_{n=0}^{N-2}||I_H(u(t^{n+2})-v^{n+2})||^2  \notag \\
        \leq& 4\Delta t\nu^{-1}\sum_{n=0}^{N-2}||f^{n+2}||_{-1}^2+2\Delta t\chi \sum_{n=0}^{N-2} ||I_Hu^{n+2}||^2 \\
        &+ \frac{2}{3}\Delta t\chi \sum_{n=0}^{N-2}||I_H(3u(t^{n+2})-4u(t^{n+1})+u(t^n))||^2 +||v^1||^2+||2v^1-v^0||^2 \notag \\
        &+\frac{2}{3}\Delta t\chi ||I_H(v^1-u(t^1))||^2+\frac{2}{3}\Delta t\chi ||2I_H(v^1-u(t^1))-I_H(v^0-u(t^0))||^2. \notag
    \end{align}
\end{theorem}

\begin{proof}
    Set $v_h=\Tilde{v}^{n+2}$ in (4.1) and $q_h=q^{n+2}$ in (4.2). Adding these two equations gives
    \begin{align}
        &\Bigg(\frac{3\Tilde{v}^{n+2}-4v^{n+1}+v^n}{2\Delta t}, \Tilde{v}^{n+2}\Bigg) + \nu ||\nabla \Tilde{v}^{n+2}||^2 =(f^{n+2},\Tilde{v}^{n+2}).
    \end{align}
    We rewrite the first term as follows
    \begin{align}
        \Bigg(\frac{3\Tilde{v}^{n+2}-4v^{n+1}+v^n}{2\Delta t}, \Tilde{v}^{n+2}\Bigg)=
        &\Bigg(\frac{3v^{n+2}-4v^{n+1}+v^n}{2\Delta t}, v^{n+2}\Bigg) +\Bigg(\frac{3\Tilde{v}^{n+2}-3v^{n+2}}{2\Delta t}, \Tilde{v}^{n+2}\Bigg) \notag \\
        &+\Bigg(\frac{3\Tilde{v}^{n+2}-4v^{n+1}+v^n}{2\Delta t}, \Tilde{v}^{n+2}-v^{n+2}\Bigg).
    \end{align}
    For the first term on the RHS in (4.12), we apply Lemma 2.1. We apply the polarization identity to the second term, and we apply (4.5) of Lemma 4.1 to the third term. Then,
    \begin{align}
        &\frac{1}{4\Delta t}\Bigg[||v^{n+2}||^2-||v^{n+1}||^2+||2v^{n+2}-v^{n+1}||^2-||2v^{n+1}-v^n||^2+||v^{n+2}-2v^{n+1}+v^n||^2\Bigg] \notag \\
        &+ \frac{1}{3}\chi(I_H(v^{n+2}-u(t^{n+2})),3v^{n+2}-4v^{n+1}+v^n)+\frac{3}{4\Delta t}(||\Tilde{v}^{n+2}||^2-||v^{n+2}||^2+||\Tilde{v}^{n+2}-v^{n+2}||^2) \notag \\
        &+\nu ||\nabla \Tilde{v}^{n+2}||^2=(f^{n+2},\Tilde{v}^{n+2}). 
    \end{align}
    Multiplying (4.13) by $4\Delta t$ and using (4.4) of Lemma 4.1 yield
    \begin{align}
        &||v^{n+2}||^2-||v^{n+1}||^2+||2v^{n+2}-v^{n+1}||^2-||2v^{n+1}-v^n||^2+||v^{n+2}-2v^{n+1}+v^n||^2 \notag \\
        &+ \frac{4}{3}\Delta t\chi(I_H(v^{n+2}-u(t^{n+2})),3v^{n+2}-4v^{n+1}+v^n)+6||\Tilde{v}^{n+2}-v^{n+2}||^2  \\
        &-4\Delta t\chi(I_H(u(t^{n+2})-v^{n+2}),v^{n+2})+4\Delta t\nu ||\nabla \Tilde{v}^{n+2}||^2=4\Delta t(f^{n+2},\Tilde{v}^{n+2}). \notag
    \end{align}
    Since $I_H$ is an $L^2$ projection, we have
    \begin{align}
        4\Delta t\chi &(I_H(u(t^{n+2})-v^{n+2}),v^{n+2}) \notag \\
        &=4\Delta t\chi(I_H(u(t^{n+2})-v^{n+2}),I_Hv^{n+2})\notag \\
        &=4\Delta t\chi(I_H(u(t^{n+2})-v^{n+2}),I_H(v^{n+2}-u^{n+2}+u^{n+2})) \notag \\
        &=-4\Delta t\chi||I_H(u(t^{n+2})-v^{n+2})||^2+4\Delta t\chi(I_H(u(t^{n+2})-v^{n+2}),I_Hu^{n+2}).
    \end{align}
    We bound the second term on the RHS using Cauchy–Schwarz and Young's inequalities as follows:
    \begin{align}
        4\Delta t\chi(I_H(u(t^{n+2})-v^{n+2}),I_Hu^{n+2}) &\leq 4\Delta t\chi ||I_H(u(t^{n+2})-v^{n+2})||^2 \cdot ||I_Hu^{n+2}||^2 \notag \\
        &\leq 4\Delta t\chi \Bigg(\frac{1}{2}||I_H(u(t^{n+2})-v^{n+2})||^2+\frac{1}{2}||I_Hu^{n+2}||^2\Bigg) \notag \\
        &=2\Delta t\chi ||I_H(u(t^{n+2})-v^{n+2})||^2+2\Delta t\chi ||I_Hu^{n+2}||^2.
    \end{align}
    Similarly, we bound the body force term as follows:
    \begin{align}
        4\Delta t(f^{n+2},\Tilde{v}^{n+2})
        &\leq 4\Delta t||f^{n+2}||_{-1}\cdot ||\nabla \Tilde{v}^{n+2}|| \leq \frac{4\Delta t}{\nu}||f^{n+2}||_{-1}^2+\Delta t\nu||\nabla \Tilde{v}^{n+2}||^2.
    \end{align}
    Lastly, we bound the sixth term in (4.14) as follows:
    \begin{align}
        &\frac{4}{3}\Delta t\chi(I_H(v^{n+2}-u(t^{n+2})),3v^{n+2}-4v^{n+1}+v^n) \notag \\
        &=\frac{4}{3}\Delta t\chi(I_H(v^{n+2}-u(t^{n+2})),I_H(3v^{n+2}-4v^{n+1}+v^n)) \notag \\
        &=\frac{2}{3}\Delta t\chi \cdot 2(I_H(v^{n+2}-u(t^{n+2})),3I_H(v^{n+2}-u(t^{n+2}))-4I_H(v^{n+1}-u(t^{n+1}))+I_H(v^n-u(t^n))) \notag \\
        &+\frac{4}{3}\Delta t\chi \cdot 2(I_H(v^{n+2}-u(t^{n+2})),3I_Hu(t^{n+2})-4I_Hu(t^{n+1})+I_Hu(t^n)) .
    \end{align}
    We apply Lemma 2.1 to the first term on the RHS in (4.18) and bound the second term by Cauchy-Schwarz-Young. Rearranging terms, we have
    \begin{align}
        &||v^{n+2}||^2-||v^{n+1}||^2+||2v^{n+2}-v^{n+1}||^2-||2v^{n+1}-v^n||^2+||v^{n+2}-2v^{n+1}+v^n||^2 \notag \\
        &+\frac{2}{3}\Delta t\chi\Bigg[||I_H(v^{n+2}-u(t^{n+2}))||^2-||I_H(v^{n+1}-u(t^{n+1}))||^2 \notag \\
        &+||2I_H(v^{n+2}-u(t^{n+2}))-I_H(v^{n+1}-u(t^{n+1}))||^2-||2I_H(v^{n+1}-u(t^{n+1}))-I_H(v^n-u(t^n))||^2 \notag \\
        &+||I_H(v^{n+2}-u(t^{n+2}))-2I_H(v^{n+1}-u(t^{n+1}))+I_H(v^n-u(t^n))||^2\Bigg]  \\
        &+6||\Tilde{v}^{n+2}-v^{n+2}||^2+3\Delta t\nu ||\nabla \Tilde{v}^{n+2}||^2 + \frac{4}{3}\Delta t\chi ||I_H(u(t^{n+2})-v^{n+2})||^2 \notag \\
        &\leq \frac{4\Delta t}{\nu}||f^{n+2}||_{-1}^2+2\Delta t\chi ||I_Hu^{n+2}||^2+\frac{2}{3}\Delta t\chi ||I_H(3u(t^{n+2})-4u(t^{n+1})+u(t^n))||^2. \notag
    \end{align}
    Summing (4.19) from $n=0$ to $N-2$ and rearranging terms completes the proof.
\end{proof}

Remark: In Theorem 4.2, the assimilated solution's estimate decomposes as follows.
\begin{align}
    \text{Discrete energy: } E^N=&||v^N||^2+||2v^N-v^{N-1}||^2+\frac{2}{3}\Delta t\chi ||I_H(v^N-u(t^N))||^2 \notag \\
    &+\frac{2}{3}\Delta t\chi ||2I_H(v^N-u(t^N))-I_H(v^{N-1}-u(t^{N-1}))||^2 \notag \\
    \text{Numerical dissipation: } D_{\text{n}}^N=&\sum_{n=0}^{N-2}||v^{n+2}-2v^{n+1}+v^n||^2 +6\sum_{n=0}^{N-2}||\Tilde{v}^{n+2}-v^{n+2}||^2 \notag \\
    \text{Viscous dissipation: } D_{\text{v}}^N=& 3\Delta t\nu \sum_{n=0}^{N-2}||\nabla \Tilde{v}^{n+2}||^2 \notag \\
    \text{Dissipation of data mismatch: }D_{\text{d}}^N=&\frac{4}{3}\Delta t\chi \sum_{n=0}^{N-2}||I_H(u(t^{n+2})-v^{n+2})||^2 \notag
\end{align}

In the last proof, we proved a finite-time stability estimate for the BDF2 algorithm. The stability estimate shows the natural energy at time $N$ and the numerical, discrete temporal, viscous, and observable error dissipation are bounded for all time by external forcing, observation data, and initial energy. Next, we analyze the long-time stability of the algorithm.

\begin{corollary}
    Consider the forecast step in (4.1) and the analysis step in (4.3). Assume $I_H$ is an $L^2$ projection. Then the time-averaged velocity is stable:
    \begin{align}
        \limsup_{N\rightarrow \infty} \frac{1}{N} \sum_{n=0}^{N-2} \Bigg[&||v^{n+2}-2v^{n+1}+v^n||^2 +6||\Tilde{v}^{n+2}-v^{n+2}||^2 + 3\Delta t\nu||\nabla \Tilde{v}^{n+2}||^2 \notag \\
        &+\frac{4}{3}\Delta t\chi||I_H(u(t^{n+2})-v^{n+2})||^2 \Bigg] \notag \\
        \leq \limsup_{N\rightarrow \infty}\frac{1}{N} \sum_{n=0}^{N-2} \Bigg[&4\Delta t\nu^{-1}||f^{n+2}||_{-1}^2+2\Delta t\chi||I_Hu^{n+2}||^2 \\
        &+ \frac{2}{3}\Delta t\chi||I_H(3u(t^{n+2})-4u(t^{n+1})+u(t^n))||^2\Bigg]. \notag
    \end{align}
\end{corollary}
\begin{proof}
    We start with (4.10). We divide both sides of inequality by $N$ and take limit as $N\rightarrow \infty$. Individual velocity terms approach zero for large $N$. Thus, we have summation terms remaining on both sides.
\end{proof}
By this corollary, we know the time-averaged dissipation and stabilization terms are bounded uniformly in time by time-averaged external forcing and data terms.

\section{Error Analysis}
  In this section, we prove two error estimates for the 2-step nudging with the analysis step in (1.5). One estimate imposes no condition on the time step size, while another estimate imposes a small time step size condition with restrictions on $\chi$. Before we prove these estimates, we introduce a few lemmas necessary for proving the error estimates.

\begin{proposition}
    Consider the analysis step in (4.3). Then the following identities holds:
    \begin{equation}
        ||e^{n+2}||^2+||e^{n+2}-\tilde{e}^{n+2}||^2+\frac{4}{3}\Delta t\chi ||I_He^{n+2}||^2= ||\tilde{e}^{n+2}||^2
    \end{equation}
    and
    \begin{equation}
        ||I_He^{n+2}||^2+||I_H(e^{n+2}-\tilde{e}^{n+2})||^2+\frac{4}{3}\Delta t\chi ||I_He^{n+2}||^2= ||I_H\tilde{e}^{n+2}||^2.
    \end{equation}
\end{proposition}

\begin{proof}
    At time $t^{n+2}$, the true solution $u$ satisfies
    \begin{align}
        \frac{3u^{n+2}-3u^{n+2}}{2\Delta t} -\chi I_H(u(t^{n+2})-u^{n+2})=0.
    \end{align}
    Subtract (1.5) from (5.3) and take inner product with $e^{n+2}$ gives
    \begin{align}
        \Bigg(\frac{3e^{n+2}-3\Tilde{e}^{n+2}}{2\Delta t}, e^{n+2}\Bigg)
        &+\chi (I_He^{n+2}, e^{n+2})=0.
    \end{align}
    For the first term on LHS of (5.4), we apply polarization identity. For the second term, we apply $I_H$ to right side of the inner product. Then
    \begin{align}
        \frac{3}{2\Delta t}\Bigg(\frac{1}{2}||e^{n+2}||^2-\frac{1}{2}||\tilde{e}^{n+2}||^2+\frac{1}{2}||e^{n+2}-\tilde{e}^{n+2}||^2\Bigg)+\chi||I_H e^{n+2}||^2=0.
    \end{align}
    Rearranging terms completes the proof of (5.1). To prove (5.2), take inner product with $I_He^{n+2}$ and apply $I_H$ to the left side of the first term. Then
    \begin{align}
        \frac{3}{2\Delta t}(I_H(e^{n+2}-\Tilde{e}^{n+2}), I_He^{n+2})
        &+\chi (I_He^{n+2}, I_He^{n+2})=0.
    \end{align}
    We apply polarization identity to the first term. We have
    \begin{align}
        \frac{3}{2\Delta t}\Bigg(\frac{1}{2}||I_He^{n+2}||^2-\frac{1}{2}||I_H\tilde{e}^{n+2}||^2+\frac{1}{2}||I_H(e^{n+2}-\tilde{e}^{n+2})||^2\Bigg)+\chi||I_H e^{n+2}||^2=0.
    \end{align}
    Rearranging terms completes the proof of (5.2).
\end{proof}

Remark: From (5.1) of Proposition 5.1, we know $||e^{n+2}||^2=||\tilde{e}^{n+2}||^2$ only if $||e^{n+2}-\tilde{e}^{n+2}||^2=||I_He^{n+2}||^2=0$. With (5.2) of Proposition 5.1, we know $||I_He^{n+2}||^2=||I_H\tilde{e}^{n+2}||^2$ only when $||I_H(e^{n+2}-\tilde{e}^{n+2})||^2=0$ and $(1+\frac{4}{3}\Delta t \chi)||I_He^{n+2}||^2=||I_H\tilde{e}^{n+2}||^2$.

\begin{corollary}
    Consider the analysis step in (4.3). Then the following identity holds:
    \begin{align}
        ||e^{n+2}||^2< ||\tilde{e}^{n+2}||^2
    \end{align} with equality when $e^{n+2}=\tilde{e}^{n+2}$ and $I_He^{n+2}=0.$
\end{corollary}

\begin{proof}
    Take (5.1). Since all terms are positive and there are more terms on LHS, we have $||e^{n+2}||^2<||\tilde{e}^{n+2}||^2$. To obtain equality, we need $||e^{n+2}-\tilde{e}^{n+2}||^2=0$ and $||I_He^{n+2}||^2=0$. This only happens when $e^{n+2}=\tilde{e}^{n+2}$ and $I_He^{n+2}=0$.
\end{proof}

\begin{corollary}
    Consider the analysis step in (4.3). Then the following identity holds:
    \begin{align}
        ||I_He^{n+2}||^2< ||I_H\tilde{e}^{n+2}||^2
    \end{align}
    with equality when $I_He^{n+2}=I_H\tilde{e}^{n+2}=0$.
\end{corollary}

\begin{proof}
    Take (5.2). Since all terms are positive and there are more terms on LHS, we have $||I_He^{n+2}||^2<||I_H\tilde{e}^{n+2}||^2$. To obtain equality, we need $||I_H(e^{n+2}-\tilde{e}^{n+2})||^2=0$ and $||I_He^{n+2}||^2=0$. This only happens when $I_He^{n+2}=I_H\tilde{e}^{n+2}=0$.
\end{proof}

\begin{lemma}
    Consider the analysis step in (4.3). If $I_H$ is $L^2$ orthogonal projection, then the following identity holds:
    \begin{equation}
        ||I_H\tilde{e}^{n+2}||^2=\Bigg(1+\frac{2}{3}\Delta t\chi\Bigg)^2||I_He^{n+2}||^2.
    \end{equation}
\end{lemma}

\begin{proof}
    At time $t^{n+2}$, the error equation is
    \begin{align}
        \Bigg(\frac{3e^{n+2}-3\Tilde{e}^{n+2}}{2\Delta t},v_h\Bigg) +\chi (I_He^{n+2}, v_h)=0.
    \end{align}
    We choose $v_h=I_He^{n+2}$. Then
    \begin{align}
        \frac{3}{2\Delta t}(e^{n+2}-\tilde{e}^{n+2}, I_He^{n+2})+\chi||I_He^{n+2}||^2=0.
    \end{align}
    We apply $I_H$ projection to the first term and rearrange terms to get
    \begin{align}
        (I_H\tilde{e}^{n+2}, I_He)=\Bigg(1+\frac{2}{3}\Delta t\chi \Bigg)||I_He^{n+2}||^2.
    \end{align}
    Now we choose $v_h=I_H\tilde{e}^{n+2}$. Then
    \begin{align}
        \frac{3}{2\Delta t}(e^{n+2}-\tilde{e}^{n+2}, I_H\tilde{e}^{n+2})+\chi(I_He^{n+2}, I_H\tilde{e}^{n+2})=0.
    \end{align}
    We again apply $I_H$ projection to the first term and rearrange terms to get
    \begin{align}
        (I_He^{n+2}, I_H\tilde{e}^{n+2})=\frac{1}{1+\frac{2}{3}\Delta t \chi}||I_H\tilde{e}^{n+2}||^2.
    \end{align}
    Due to symmetry, we can combine (5.13) and (5.15) to yield
    \begin{align}
        \Bigg(1+\frac{2}{3}\Delta t \chi\Bigg)||I_He^{n+2}||^2=\frac{1}{1+\frac{2}{3}\Delta t\chi}||I_H\tilde{e}^{n+2}||^2.
    \end{align}
    Multiplying through by $\Bigg(1+\frac{2}{3}\Delta t\chi\Bigg)$ completes the proof.
\end{proof}

The following lemma is crucial for handling the time discretization term in the error analysis.

\begin{lemma}
    Consider the analysis step in (4.3). Then the following identity holds:
    \begin{align}
        &\Bigg(\frac{3e^{n+2}-4e^{n+1}+e^n}{2\Delta t}, \Tilde{e}^{n+2}-e^{n+2}\Bigg) \notag \\
        =&\frac{1}{6}\chi \Bigg[ ||I_He^{n+2}||^2-||I_He^{n+1}||^2 + ||2I_He^{n+2}-I_He^{n+1}||^2 - ||2I_He^{n+1}-I_He^n||^2 \\
        &+ ||I_He^{n+2}-2I_He^{n+1}+I_He^n||^2\Bigg]. \notag
    \end{align}
\end{lemma}

\begin{proof}
    At time $t^{n+2}$, the true solution $u$ satisfies
    \begin{align}
        \frac{3u^{n+2}-3u^{n+2}}{2\Delta t} -\chi I_H(u(t^{n+2})-u^{n+2})=0.
    \end{align}
    Subtract (1.5) from (5.18) and take inner product with $\frac{3e^{n+2}-4e^{n+1}+e^n}{3}$ gives
    \begin{align}
        \Bigg(\frac{3e^{n+2}-3\Tilde{e}^{n+2}}{2\Delta t}, \frac{3e^{n+2}-4e^{n+1}+e^n}{3}\Bigg)
        &+\chi \Bigg(I_He^{n+2}, \frac{3e^{n+2}-4e^{n+1}+e^n}{3}\Bigg)=0. \notag
    \end{align}
    For the second term, we apply $I_H$ to right side of the inner product and use Lemma 2.1. Rearranging terms completes the proof.
\end{proof}

We now present the first error estimate in Theorem 5.4.

\begin{theorem}
    Consider forecast step in (1.3) and analysis step in (1.5). Suppose $I_H$ is an $L^2$ projection. Assume 
    \begin{align}
        \frac{31}{8}\nu-4\frac{U}{l}C_P^2L^2>0 \text{, equivalently, } Re=\frac{UL}{\nu}<\frac{31}{32}\frac{l}{C_P^2L}. \notag
    \end{align} Then the following holds:
    \begin{align}
        ||e^N||^2
        &+||2e^N-e^{N-1}||^2+\frac{2}{3}\Delta t\chi ||I_He^N||^2+\frac{2}{3}\Delta t\chi ||2I_He^N-I_He^{N-1}||^2 \notag \\
        &+\sum_{n=0}^{N-2}||e^{n+2}-2e^{n+1}+e^n||^2 + \frac{2}{3}\Delta t\chi \sum_{n=0}^{N-2} ||I_H(e^{n+2}-2e^{n+1}+e^n)||^2 \notag \\
        &+ 6 \sum_{n=0}^{N-2} ||\tilde{e}^{n+2}-e^{n+2}||^2 +\Delta t\Bigg(\frac{31}{8}\nu -4\frac{U}{l}C_P^2L^2\Bigg) \sum_{n=0}^{N-2}||\nabla \tilde{e}^{n+2}||^2 \\
        \leq \frac{8}{\nu}&C\Delta t^4\int_{t^0}^{t^{N}}||u_{ttt}||_{-1}^2 + ||e^1||^2+||2e^1-e^0||^2 \notag \\
        &+\frac{2}{3}\Delta t\chi ||I_He^1||^2+\frac{2}{3}\Delta t\chi ||2I_He^1-I_He^0||^2. \notag
    \end{align}
\end{theorem}

\begin{proof}
    For the forecast step, at time $t^{n+2}$, the true solution $u,p$ satisfies
    \begin{align}
        \frac{3u^{n+2}-4u^{n+1}+u^n}{2\Delta t}+u^{n+2} \cdot \nabla u^{n+2} - \nu \Delta u^{n+2} + \nabla p^{n+2}=f^{n+2}+\rho^{n+2},
    \end{align}
    where $\rho^{n+2}=\frac{3u^{n+2}-4u^{n+1}+u^n}{2\Delta t}-u_t(t^{n+2})$ is the consistency error term.
    
    Subtracting (1.3) from (5.20) and taking inner product with $\Tilde{e}^{n+2}$ yields
    \begin{align}
        \Bigg(\frac{3\Tilde{e}^{n+2}-4e^{n+1}+e^n}{2\Delta t}, \Tilde{e}^{n+2}\Bigg)
        &+(u^{n+2} \cdot \nabla u^{n+2}- \tilde{v}^{n+2} \cdot \nabla \Tilde{v}^{n+2}, \Tilde{e}^{n+2}) \notag \\
        &+ \nu ||\nabla \Tilde{e}^{n+2}||^2=(\rho^{n+2}, \Tilde{e}^{n+2}).
    \end{align}
    We rewrite the first term exactly like in (4.12). Now, the first term is split into three terms. For the split terms, we apply Lemma 2.1 to the first term, polarization identity to the second term, and (5.17) of Lemma 5.3 to the third term. Then, we have
    \begin{align}
        &\frac{1}{4\Delta t}\Bigg[||e^{n+2}||^2-||e^{n+1}||^2+||2e^{n+2}-e^{n+1}||^2-||2e^{n+1}-e^n||^2+||e^{n+2}-2e^{n+1}+e^n||^2\Bigg] \notag   \\
        &+\frac{1}{6}\chi \Bigg[ ||I_He^{n+2}||^2 - ||I_He^{n+1}||^2 + ||2I_He^{n+2}-I_He^{n+1}||^2 - ||2I_He^{n+1}-I_He^n||^2  \notag \\
        &+||I_He^{n+2}-2I_He^{n+1}+I_He^n||^2\Bigg] \\
        &+\frac{3}{4\Delta t}\Bigg[||\Tilde{e}^{n+2}||^2-||e^{n+2}||^2 + ||\Tilde{e}^{n+2}-e^{n+2}||^2 \Bigg] \notag \\
        &+(u^{n+2} \cdot \nabla u^{n+1}- \tilde{v}^{n+2} \cdot \nabla \Tilde{v}^{n+2}, \Tilde{e}^{n+2}) +\nu ||\nabla \Tilde{e}^{n+2}||^2 = (\rho^{n+2},\Tilde{e}^{n+2}). \notag
    \end{align}
    We handle the nonlinear term by adding and subtracting $\tilde{v}^{n+2}\cdot \nabla u^{n+2}$ as follows
    \begin{align}
        &(u^{n+2} \cdot \nabla u^{n+2}-\tilde{v}^{n+2} \cdot \nabla \Tilde{v}^{n+2}, \Tilde{e}^{n+2}) \notag \\
        =&(u^{n+2} \cdot \nabla u^{n+2} -\tilde{v}^{n+2} \cdot \nabla u^{n+2} + \tilde{v}^{n+2}\cdot \nabla u^{n+2} + \tilde{v}^{n+2}\cdot \nabla \tilde{v}^{n+2}, \Tilde{e}^{n+2}) \notag \\
        =&(\tilde{e}^{n+2}\cdot \nabla u^{n+2}, \Tilde{e}^{n+2})+(\tilde{v}^{n+2}\cdot \nabla \Tilde{e}^{n+2}, \Tilde{e}^{n+2}) \notag \\
        =&(\tilde{e}^{n+2}\cdot \nabla u^{n+2}, \Tilde{e}^{n+2}).    \end{align}
    We bound (5.23) using H\"{o}lder's inequality as follows
    \begin{align}
        (\tilde{e}^{n+2}\cdot \nabla u^{n+2}, \Tilde{e}^{n+2}) 
        \leq ||\tilde{e}^{n+2}||\cdot ||\nabla u^{n+2}||_\infty \cdot ||\Tilde{e}^{n+2}|| &\leq \frac{U}{l}||\tilde{e}^{n+2}||^2 \notag \\
        &\leq \frac{U}{l}C_P^2L^2 ||\nabla \tilde{e}^{n+2}||^2,
    \end{align}
    where we use Poincar\'e for the last inequality. For the consistency term, we bound using H\"{o}lder's and Young's inequalities as follows
    \begin{align}
        (\rho^{n+2}, \Tilde{e}^{n+2})\leq ||\rho^{n+2}||_{-1} \cdot ||\nabla \Tilde{e}^{n+2}||\leq \frac{2}{\nu}||\rho^{n+2}||_{-1}^2+\frac{\nu}{8}||\nabla \Tilde{e}^{n+2}||^2.
    \end{align}
    Using Lemma 2.2, 
    \begin{align}
        ||\rho^{n+2}||_{-1}^2\leq C\Delta t^3\int_{t^n}^{t^{n+2}}||u_{ttt}||_{-1}^2dt.
    \end{align}
    For the analysis step, at time $t^{n+2}$, the true solution $u$ satisifies
    \begin{align}
        \frac{3u^{n+2}-3u^{n+2}}{2\Delta t}-\chi I_H(u(t^{n+2})-u^{n+2})=0.
    \end{align}
    Subtracting (1.5) from (5.18) and taking inner product with $e^{n+2}$ gives
    \begin{align}
        \Bigg(\frac{3e^{n+2}-3\Tilde{e}^{n+2}}{2\Delta t}, e^{n+2}\Bigg)+\chi (I_He^{n+2},e^{n+2})=0.
    \end{align}
    We apply the polarization identity to the first term. For the second term, we apply $I_H$ to the right side of the inner product. Rearranging terms, we have 
    \begin{align}
        ||e^{n+2}||^2=||\tilde{e}^{n+2}||^2-||e^{n+2}-\tilde{e}^{n+2}||^2-\frac{4}{3}\Delta t\chi ||I_He^{n+2}||^2.
    \end{align}
    We use (5.29) to replace $||e^{n+2}||^2$ in (5.22) and combine the resulting equality with (5.24) and (5.26). Rearranging terms, we have
    \begin{align}
        &\Bigg[||e^{n+2}||^2-||e^{n+1}||^2+||2e^{n+2}-e^{n+1}||^2-||2e^{n+1}-e^n||^2+||e^{n+2}-2e^{n+1}+e^n||^2\Bigg] \notag  \\
        &+\frac{2}{3}\Delta t\chi \Bigg[ ||I_He^{n+2}||^2 - ||I_He^{n+1}||^2 + ||2I_He^{n+2}-I_He^{n+1}||^2 - ||2I_He^{n+1}-I_He^n||^2 \notag \\
        &+||I_He^{n+2}-2I_He^{n+1}+I_He^n||^2\Bigg]  \\
        &+6||\Tilde{e}^{n+2}-e^{n+2}||^2 +4\Delta t\chi ||I_He^{n+2}||^2 +\Delta t\Bigg(\frac{31}{8}\nu -4\frac{U}{l}C_P^2L^2\Bigg) ||\nabla \Tilde{e}^{n+2}||^2 \notag \\
        &\leq \frac{8}{\nu}C\Delta t^4\int_{t^n}^{t^{n+2}}||u_{ttt}||_{-1}^2dt. \notag
    \end{align}
    Taking the sum from $n = 0$ to $N-2$
    completes the proof.
\end{proof}

\begin{corollary}
    Consider forecast step in (1.3) and analysis step in (1.5). Suppose $I_H$ is an $L^2$ projection. Assume 
    \begin{align}
        \frac{31}{8}\nu-4\frac{U}{l}C_P^2L^2>0  \text{, equivalently, } Re=\frac{UL}{\nu}<\frac{31}{32}\frac{l}{C_P^2L}. \notag
    \end{align} Then the following holds:
    \begin{align}
        ||e^N||^2\leq C&\Delta t^4\int_{t^0}^{t^{N}}||u_{ttt}||_{-1}^2 + ||e^1||^2+||2e^1-e^0||^2 \notag \\
        &+\frac{2}{3}\Delta t\chi ||I_He^1||^2+\frac{2}{3}\Delta t\chi ||2I_He^1-I_He^0||^2.
    \end{align}
\end{corollary}

\begin{proof}
    We begin with (5.19) and assume $\frac{31}{8}\nu-4\frac{U}{l}C_P^2>0$. Since all LHS terms are nonnegative, we can drop all but $||e^N||^2$. On the RHS, we have the integral term and initial errors, bounding the error at time $N$.
\end{proof}

\begin{corollary}
    Consider forecast step in (1.3) and analysis step in (1.5). Suppose $I_H$ is an $L^2$ projection. Assume 
    \begin{align}
        \frac{31}{8}\nu-4\frac{U}{l}C_P^2L^2>0 \text{, equivalently, } Re=\frac{UL}{\nu}<\frac{31}{32}\frac{l}{C_P^2L}. \notag
    \end{align} Then the following holds:
    \begin{align}
        ||e^N||=\mathcal{O}(\Delta t^2).
    \end{align}
\end{corollary}

\begin{proof}
    We begin with (5.19). We take the square root of both sides of the inequality. Since all LHS terms are nonnegative, drop all but $||e^N||^2$. On the RHS, the dominant term is the integral term with coefficient $\Delta t^2$. Thus, the error at time $N$ is second-order accurate in time. 
\end{proof}
From the two corollaries above, we know the error does not grow in time, the numerical solution remains uniformly bounded for all $N$, and the method converges second-order in time.

Now we consider a different strategy in proving the error estimate and present the small timestep estimate in Theorem 5.5.

\begin{theorem}
    Consider forecast step in (1.3) and analysis step in (1.5). Suppose $I_H$ is an $L^2$ projection. Assume 
    \begin{align}
        \nu-\frac{8}{7}\frac{U}{l}C^2H^2>0 \text{, equivalently, } Re=\frac{UL}{\nu}<\frac{7}{8}\frac{lL}{C^2H^2}. \notag
    \end{align}
    Define $T^*=\frac{L}{U}$. Assume
    \begin{align}
    \frac{\Delta t}{T^*} \leq \frac{3}{8}\frac{l}{L}. \notag
    \end{align}
    If $\chi\in [\chi_1,\chi_2]\cap(0,\infty)$, where $\chi_1,\chi_2$ are the roots of $\chi-\frac{U}{l}(1+\frac{2}{3}\Delta t\chi)^2$, then the following holds:
    \begin{align}
        ||e^N||^2
        &+||2e^N-e^{N-1}||^2+\frac{2}{3}\Delta t\chi ||I_He^N||^2+\frac{2}{3}\Delta t\chi ||2I_He^N-I_He^{N-1}||^2  \notag \\
        +\sum_{n=0}^{N-2}&||e^{n+2}-2e^{n+1}+e^n||^2 + \frac{2}{3}\Delta t\chi \sum_{n=0}^{N-2} ||I_H(e^{n+2}-2e^{n+1}+e^n)||^2 + 6 \sum_{n=0}^{N-2} ||\tilde{e}^{n+2}-e^{n+2}||^2 \notag \\
        +\Delta t&\Bigg(\frac{7}{2}\nu -4\frac{U}{l}C^2H^2\Bigg) \sum_{n=0}^{N-2}||\nabla \tilde{e}^{n+2}||^2+4\Delta t\Bigg(\chi - \frac{U}{l}(1+\frac{2}{3}\Delta t\chi)^2\Bigg)\sum_{n=0}^{N-2} ||I_He^{n+2}||^2 \\
        \leq \frac{8}{\nu}&C\Delta t^4\int_{t^0}^{t^{N}}||u_{ttt}||_{-1}^2 + ||e^1||^2+||2e^1-e^0||^2+\frac{2}{3}\Delta t\chi ||I_He^1||^2+\frac{2}{3}\Delta t\chi ||2I_He^1-I_He^0||^2. \notag
    \end{align}
\end{theorem}

\begin{proof}
    The proof starts identical to the proof of Theorem 5.4 but deviates by a different bound for (5.24).
    We bound (5.24) using H\"{o}lder's inequality as follows
    \begin{align}
        (\tilde{e}^{n+2}\cdot \nabla u^{n+2}, \Tilde{e}^{n+2}) 
        &\leq ||\tilde{e}^{n+2}||\cdot ||\nabla u^{n+2}||_\infty \cdot ||\Tilde{e}^{n+2}|| \notag \\
        &\leq \frac{U}{l}||\tilde{e}^{n+2}||^2, \text{ where } \frac{U}{l}=\max_{\Omega \times (0,\infty)}||\nabla u(x,t)||, U= \max_{\Omega \times (0,\infty)}|| u(x,t)||\notag \\
        &\leq \frac{U}{l}\Bigg[||I_H\tilde{e}^{n+2}||^2+ ||(I-I_H)\tilde{e}^{n+2}||^2\Bigg] \notag \\
        &\leq \frac{U}{l}||I_H\tilde{e}^{n+2}||^2+ \frac{U}{l}C^2H^2||\nabla \tilde{e}^{n+2}||^2,
    \end{align}
    where the last inequality follows from (2.1).
    For the consistency term, we bound using H\"{o}lder's and Young's inequalities as follows
    \begin{align}
        (\rho^{n+2}, \Tilde{e}^{n+2})\leq ||\rho^{n+2}||_{-1} \cdot ||\nabla \Tilde{e}^{n+2}||\leq \frac{2}{\nu}||\rho^{n+2}||_{-1}^2+\frac{\nu}{8}||\nabla \Tilde{e}^{n+2}||^2.
    \end{align}
    Using Lemma 2.2, 
    \begin{align}
        ||\rho^{n+2}||_{-1}^2\leq C\Delta t^3\int_{t^n}^{t^{n+2}}||u_{ttt}||_{-1}^2dt.
    \end{align}
    For the analysis step, at time $t^{n+2}$, the true solution $u$ satisifies
    \begin{align}
        \frac{3u^{n+2}-3u^{n+2}}{2\Delta t}-\chi I_H(u(t^{n+2})-u^{n+2})=0.
    \end{align}
    Subtracting (1.5) from (5.37) and taking inner product with $e^{n+2}$ gives
    \begin{align}
        \Bigg(\frac{3e^{n+2}-3\Tilde{e}^{n+2}}{2\Delta t}, e^{n+2}\Bigg)+\chi (I_He^{n+2},e^{n+2})=0.
    \end{align}
    We apply the polarization identity to the first term. For the second term, we apply $I_H$ to the right side of the inner product. Rearranging terms, we have 
    \begin{align}
        ||e^{n+2}||^2=||\tilde{e}^{n+2}||^2-||e^{n+2}-\tilde{e}^{n+2}||^2-\frac{4}{3}\Delta t\chi ||I_He^{n+2}||^2.
    \end{align}
    We use (5.39) to replace $||e^{n+2}||^2$ in (5.22) and combine the resulting equality with (5.34) and (5.36). Rearranging terms, we have
    \begin{align}
        \Bigg[||e^{n+2}||^2-||e^{n+1}||^2+&||2e^{n+2}-e^{n+1}||^2-||2e^{n+1}-e^n||^2+||e^{n+2}-2e^{n+1}+e^n||^2\Bigg] \notag  \\
        +\frac{2}{3}\Delta t\chi \Bigg[||I_He^{n+2}&||^2 - ||I_He^{n+1}||^2 + ||2I_He^{n+2}-I_He^{n+1}||^2 - ||2I_He^{n+1}-I_He^n||^2 \notag \\
        +||&I_He^{n+2}-2I_He^{n+1}+I_He^n||^2\Bigg]  \\
        +6||\Tilde{e}^{n+2}-e^{n+2}&||^2 +4\Delta t\chi ||I_He^{n+2}||^2 +\Delta t\Bigg(\frac{7}{2}\nu -4\frac{U}{l}C^2H^2\Bigg) ||\nabla \Tilde{e}^{n+2}||^2 \notag \\
        \leq \frac{8}{\nu}C\Delta t^4\int_{t^n}^{t^{n+2}}&||u_{ttt}||_{-1}^2dt + 4\Delta t\frac{U}{l}||I_H\tilde{e}^{n+2}||^2. \notag
    \end{align}
    Applying Lemma 5.2 to $||I_H\tilde{e}^{n+2}||^2$ and rearranging terms give
    \begin{align}
        \Bigg[||e^{n+2}||^2-||e^{n+1}||^2+&||2e^{n+2}-e^{n+1}||^2-||2e^{n+1}-e^n||^2+||e^{n+2}-2e^{n+1}+e^n||^2\Bigg] \notag  \\
        +\frac{2}{3}\Delta t\chi \Bigg[ ||I_He^{n+2}||^2& - ||I_He^{n+1}||^2 + ||2I_He^{n+2}-I_He^{n+1}||^2 - ||2I_He^{n+1}-I_He^n||^2 \notag \\
        &+||I_He^{n+2}-2I_He^{n+1}+I_He^n||^2\Bigg]  \\
        +6||\Tilde{e}^{n+2}-e^{n+2}||^2 &+4\Delta t\Bigg(\chi-\frac{U}{l}(1+\frac{2}{3}\Delta t \chi)^2\Bigg)||I_He^{n+2}||^2 \notag \\
        +\Delta t\Bigg(\frac{7}{2}\nu -4\frac{U}{l}C^2&H^2\Bigg) ||\nabla \Tilde{e}^{n+2}||^2 \leq \frac{8}{\nu}C\Delta t^4\int_{t^n}^{t^{n+2}}||u_{ttt}||_{-1}^2dt. \notag
    \end{align}
    Taking the sum from $n = 0$ to $N-2$ completes the proof.
\end{proof}

Remark: The conditions on $\Delta t$ and $\chi$ are derived from the fact that $\chi-\frac{U}{l}(1+\frac{2}{3}\Delta t\chi)^2$ needs to be greater or equal to zero. The derivation goes as follows. Let $A=\frac{U}{l}$ and $B=\frac{2}{3}\Delta t$. Then 
\begin{align}
    \chi-\frac{U}{l}(1+\frac{2}{3}\Delta t\chi)^2\geq 0 \iff \chi -A(1+B\chi)^2\geq 0. \notag
\end{align}
Expanding this quadratic expression, we get
\begin{align}
    AB^2\chi^2+(2AB-1)\chi+A\leq 0. \notag
\end{align}
The discriminant is $(2AB-1)^2-4A^2B^2$, and it must be greater than or equal to zero for the quadratic expression to have real roots. Then expanding the discriminant, we have
\begin{align}
    4A^2B^2-4AB+1-4A^2B^2\geq 0. \notag
\end{align}
This implies $AB\leq \frac{1}{4}$ and $B\leq \frac{1}{4A}$. This means $\frac{2}{3}\Delta t\leq \frac{l}{4U}$. Thus, we have $\Delta t\leq \frac{3}{8}\frac{l}{U}$. To make the $\Delta t$ condition dimensionless, we let $T^*=\frac{L}{U}$ and rearrange the inequality.

\begin{corollary}
    Consider forecast step in (1.3) and analysis step in (1.5). Suppose $I_H$ is an $L^2$ projection. Assume 
    \begin{align}
        \nu-\frac{8}{7}\frac{U}{l}C^2H^2>0 \text{, equivalently, } Re=\frac{UL}{\nu}<\frac{7}{8}\frac{lL}{C^2H^2}. \notag
    \end{align}
    Define $T^*=\frac{L}{U}$. Assume
    \begin{align}
    \frac{\Delta t}{T^*} \leq \frac{3}{8}\frac{l}{L}. \notag
    \end{align}
    If $\chi\in [\chi_1,\chi_2]\cap(0,\infty)$, where $\chi_1,\chi_2$ are the roots of $\chi-\frac{U}{l}(1+\frac{2}{3}\Delta t\chi)^2$, then the following holds:
    \begin{align}
        ||e^N||^2\leq C\Delta t^4\int_{t^0}^{t^{N}}||u_{ttt}||_{-1}^2 + 
        \text{initial errors}.
    \end{align}
\end{corollary}

\begin{proof}
    We begin with (5.33). By assumption, we have $\nu-\frac{8}{7}\frac{U}{l}C^2H^2>0 \text{ and }\Delta t\leq \frac{3}{8}\frac{l}{U}$. Then we can drop all LHS terms that are nonnegative and keep $||e^N||^2$. On the RHS, we have the integral term and initial errors.
\end{proof}
This corollary says that the error does not grow over time as long as the conditions on the viscosity and time step are satisfied: $\nu-\frac{8}{7}\frac{U}{l}C^2H^2>0 \text{ and }\Delta t\leq \frac{3}{8}\frac{l}{U}$.

\section{Numerical Experiments}
  In this section, we present three numerical tests. Through simulations, we test the convergence and accuracy of the BDF2, modular nudging algorithm. For all tests, we implement the explicit form (3.8) of the analysis step.
Recall that the explicit analysis step holds if $I_H=I_H^2$ but may lose accuracy if that condition is not satisfied. An alternative, more expensive implementation is a direct solve for the analysis step to avoid losing accuracy.
\subsection{Experiment 1: Temporal Convergence Test}
The first test comes from Cibik, Fang, Layton, and Siddiqua \cite{Cibik:2025}. We verify the temporal convergence rate of modular, 2-step nudging with BDF2 time discretization by considering an analytical solution on unit square domain $\Omega=(0,1)^2$. The analytical velocity and pressure are
\begin{align}
    u(x,y,t)=e^t(\cos y, \sin x)^T \text{ and } p(x,y,t)=(x-y)(1+t). \notag
\end{align}
The forcing term $f(x,t)$ is computed using the analytical solutions. To minimize spatial errors, we generate a fine mesh with 43,266 degrees of freedom (dof). We assume $I_H$ is $L^2$ projection onto a continuous piecewise quadratic velocity space. We choose Scott-Vogelius $(P2-P1dc)$ finite element pair with a barycenter refined mesh. The results are recorded for errors in $L^2(\Omega)$ norm, convergence rates, and CPU runtime in seconds. The convergence rate is calculated using the formula
\begin{align}
    \text{rate}=\frac{\log(e_{\Delta t_1}/e_{\Delta t_2})}{\log(\Delta t_1/\Delta t_2)}. \notag
\end{align}
Further, we compare with standard nudging with BDF2 time discretization. We set $\chi=1$. The numerical results for final time T=4 are as follows:

\begin{table}[H]
\centering
\begin{tabular}{||c|cc|cc||}
\toprule
 & \multicolumn{2}{c|}{2-step nudging}
 & \multicolumn{2}{c||}{Standard nudging} \\
\cmidrule(lr){2-3} \cmidrule(lr){4-5}
$\Delta t$
& $\dfrac{\|u-v\|}{\|u\|}$ & rate
& $\dfrac{\|u-v\|}{\|u\|}$ & rate \\
\midrule
1     & 2.15e-4 & -- & 2.66e-4 & -- \\
1/2   & 6.89e-5 & 1.64 & 7.23e-5 & 1.88 \\
1/4   & 1.96e-5 & 1.81 & 1.96e-5 & 1.87 \\
1/8   & 5.25e-6 & 1.90 & 5.21e-6 & 1.91 \\
1/16  & 1.35e-6 & 1.95 & 1.34e-6 & 1.95 \\
1/32  & 3.64e-7 & 1.89 & 3.63e-7 & 1.89 \\
\bottomrule
\end{tabular}
\caption{We observe second-order temporal convergence for modular, 2-step nudging.}
\end{table}

As shown in Table 1, we observe second-order convergence rate for modular, 2-step nudging. Standard nudging solves one coupled system at each time step, while modular nudging decouples the system into two cheaper steps. Modular nudging has second-order temporal accuracy and preserves stability.

To illustrate the reduction in computational complexity of modular nudging, we run the simulations for $\chi=1$ and $\chi=10^4$ and compare the final CPU runtime.

\begin{table}[H]
\centering
\begin{tabular}{||c|c|c||}
\toprule
 & \multicolumn{1}{c|}{2-step nudging}
 & \multicolumn{1}{c||}{Standard nudging} \\
\cmidrule(lr){2-2} \cmidrule(lr){3-3}
$\Delta t$
& CPU time (s)
& CPU time (s) \\
\midrule
1     & 4.61 & 11.66 \\
1/2   & 10.63 & 23.65 \\
1/4   & 22.57 & 47.79 \\
1/8   & 45.99 & 92.46 \\
1/16  & 93.33 & 179.76 \\
1/32  & 189.09 & 366.11 \\
\bottomrule
\end{tabular}
\caption{($\chi=1$) Modular
nudging has lower computational complexity than standard nudging.}
\end{table}

\begin{table}[H]
\centering
\begin{tabular}{||c|c|c||}
\toprule
 & \multicolumn{1}{c|}{2-step nudging}
 & \multicolumn{1}{c||}{Standard nudging} \\
\cmidrule(lr){2-2} \cmidrule(lr){3-3}
$\Delta t$
& CPU time (s)
& CPU time (s) \\
\midrule
1     & 4.46 & 14.41 \\
1/2   & 10.42 & 28.69 \\
1/4   & 22.43 & 58.92 \\
1/8   & 46.08 & 119.35 \\
1/16  & 93.97 & 229.90 \\
1/32  & 187.92 & 460.16 \\
\bottomrule
\end{tabular}
\caption{($\chi=10^4$) Modular
nudging has lower computational complexity than standard nudging.}
\end{table}

Table 2 and Table 3 show the CPU runtime in seconds of computing numerical solutions for modular and standard nudging at the final time $T=4$. The recorded CPU runtime serves as a representative measure for runtime varies by program execution. For all time steps, modular nudging takes roughly half the time needed to compute the solution than standard nudging. This suggests modular nudging reduces computational complexity. The results are computed using FreeFEM software, but CPU runtime may be similar for other software, i.e, NGSolve. 

\subsection{Experiment 2: Flow between offset cylinders}
The second test is adapted from Cibik, Fang, and Layton \cite{Fang:2025}. We evaluate the performance of modular, 2-step nudging with a complex 2d flow at a higher Reynolds number without exact solutions. We consider performance comparison with standard nudging and without nudging. The domain is a disk containing a smaller, off-centered obstacle. The disk remains fixed in place. Let $r_1=1$ be the outer circle radius and $r_2=0.1$ be the inner circle radius with center at $c=(1/2,0)$. Formally, the domain is
\begin{align}
    \Omega=\{(x,y):x^2+y^2\leq r_1^2 \text{ and } (x-c_1)^2+(y-c_2)^2\geq r_2^2\}. \notag
\end{align}
We assume no-slip boundary conditions, and the flow is driven by a rotational force
\begin{align}
    f(x,y,t)=(-4y\min(1,t)(1-x^2-y^2), 4x\min(1,t)(1-x^2-y^2))^T. \notag
\end{align}

Using the Delaunay algorithm, we use a mesh that has 75 mesh points on the outer boundary and 60 mesh points on the inner boundary. We set final time $T=25$, time step size $\Delta t=\frac{1}{100}, \nu=10^{-3}, L=1, U=1,$ and $Re=\frac{LU}{\nu}$. The initial conditions are $u(x,y,0)=0$ and $v_0=I_H(u_0)$. We impose the Dirichlet boundary condition $u=0$ on $\partial \Omega$.

We apply Taylor-Hood $(P2-P1)$ finite element pair for the velocity and pressure spaces. We treat the nonlinear convection term using a skew-symmetric formulation with divergence correction that takes the form
\begin{align}
    (v^{n+1}\cdot \nabla \Tilde{v}^{n+2}, w)+\frac{1}{2}((\nabla \cdot v^{n+1})\Tilde{v}^{n+2}, w), \text{ $\forall w \in X^H$}. \notag
\end{align}

The true solution $u$ is approximated using Direct Numerical Simulation (DNS), calculated on a finer mesh with 120 mesh points on the outer boundary and 96 mesh points on the inner boundary. The DNS solution uses the BDF2 time discretization scheme. The nonlinear term is treated using the skew-symmetric formulation with divergence correction. We set $I_H$ to be the $L^2$ projection and set $\chi=1$ and $10^4$.

Figures are plotted starting the second step to exclude the initial error being undefined on log scale.

\begin{figure}[H]
    \centering
    \includegraphics[width=0.8\linewidth]{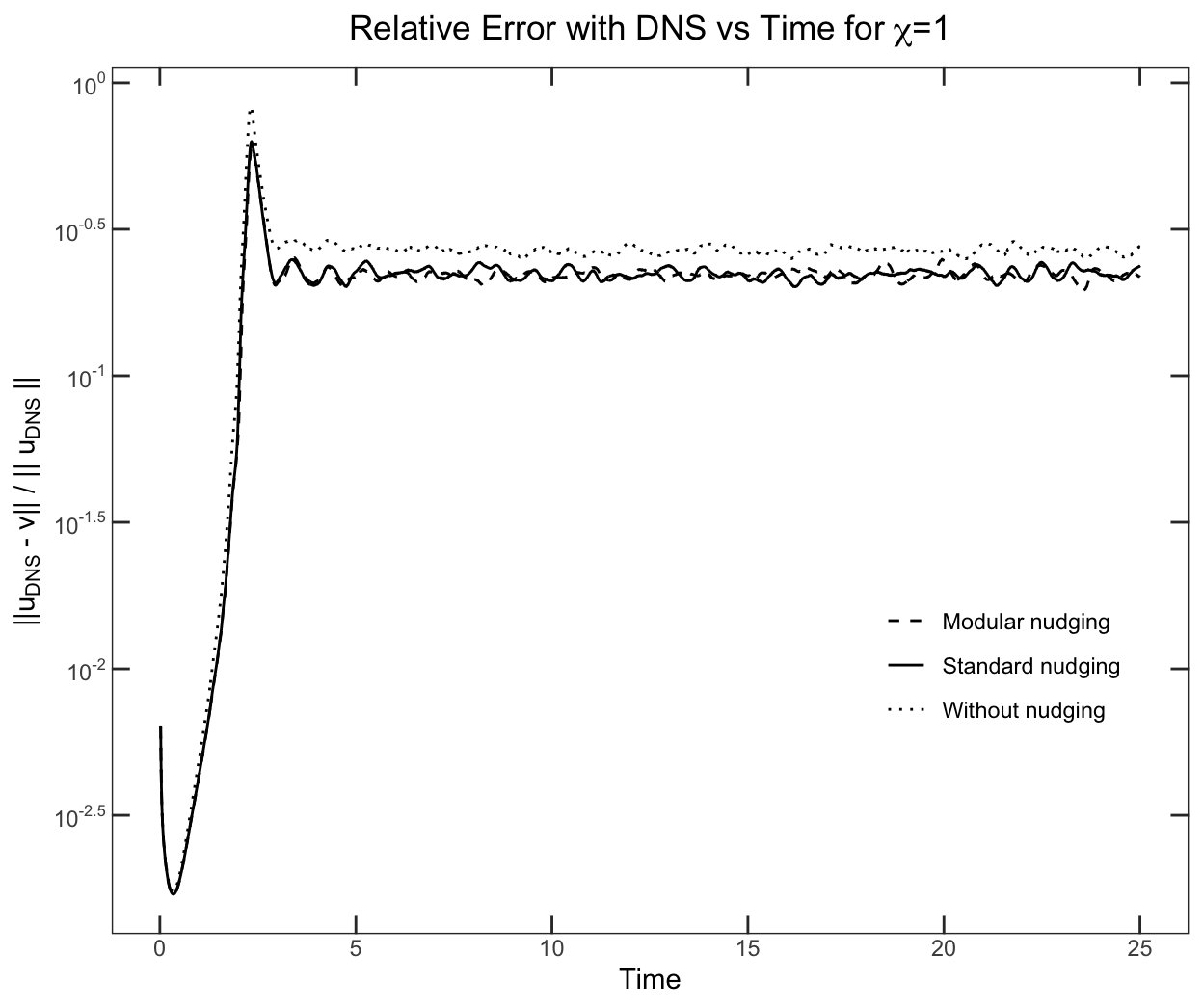}
    \caption{For $\chi=1$, modular nudging performs similar to standard nudging but better than without nudging.}
    \label{fig:Test2X=1}
\end{figure}

For $\chi=1$, the relative errors for modular, standard, and without nudging  decrease to around $\mathcal{O}(10^{-2.5})$ from $t=0$ to $t=1$. After $t=1$, the errors sharply increase and saturate at around $\mathcal{O}(10^{-0.5})$. We observe comparable longtime performance for modular and standard nudging, with both performing better than without nudging. We conjecture that the sharp growth in relative error at $t\approx1$ is due to $H$-condition not being satisfied.

\begin{figure}[H]
    \centering
    \includegraphics[width=0.8\linewidth]{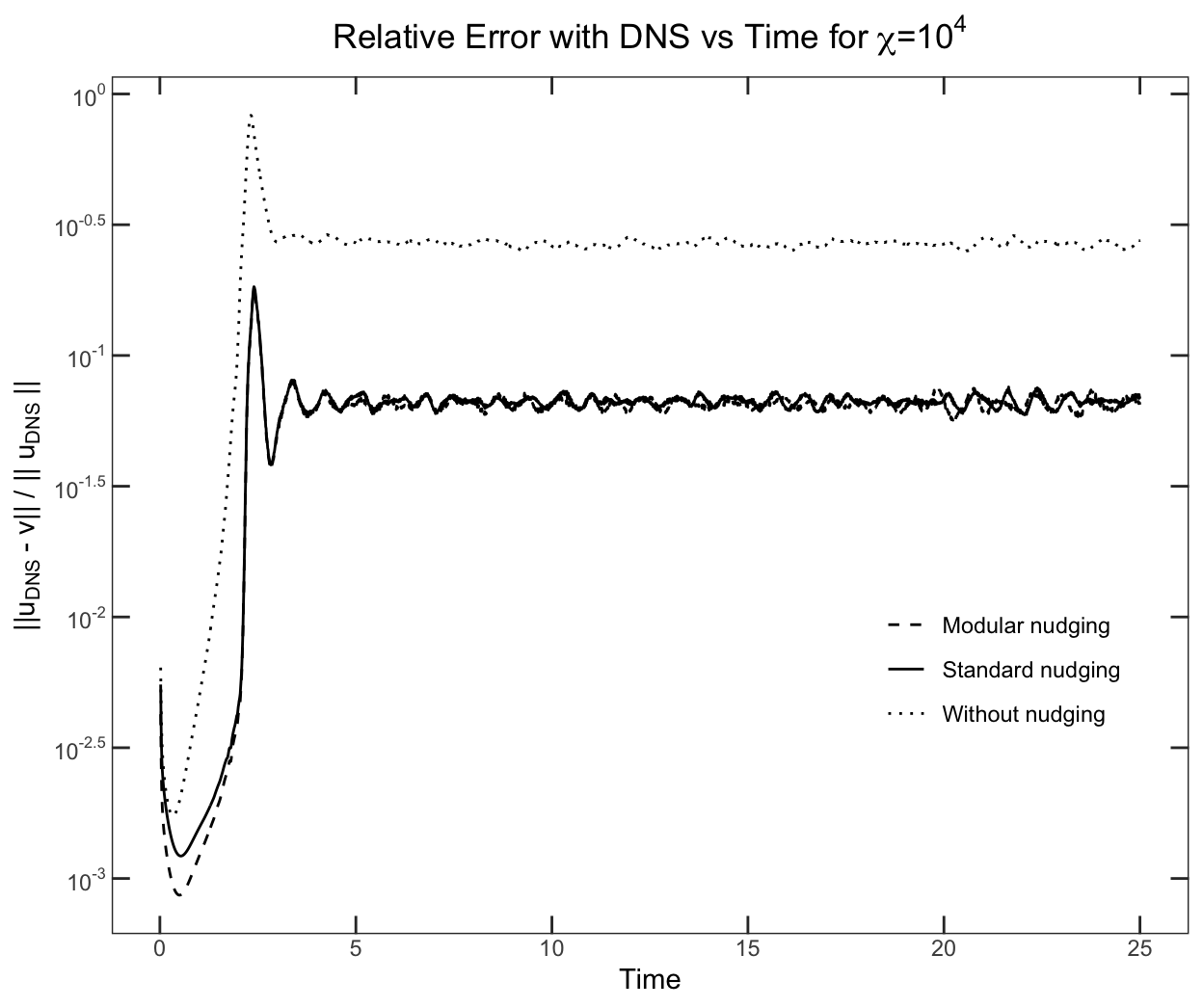}
    \caption{For $\chi=10^4$, modular nudging performs better than standard nudging in short time but similar in long time, both performing better than without nudging.}
    \label{fig:Test2X=1e4}
\end{figure}

For $\chi=10^4$, the relative error for modular, standard, and without nudging decrease to $\mathcal{O}(10^{-3})$ from $t=0$ to $t=1$, with modular nudging performing the best. After $t=1$, errors for modular and standard nudging increase and saturate at around $\mathcal{O}(10^{-1})$. Without nudging, the errors increase and saturate at around $\mathcal{O}(10^{-0.5})$ after $t=1$. For this test, modular nudging performed the best in short time and performed similar to standard nudging in long time, while without nudging performed the worst. 

Compared with results for modular nudging with backward Euler (BE) method for $\chi=1$ in the referenced test \cite{Fang:2025}, we observe comparable long-time performance and better short-time performance when using the BDF2 method. For $\chi=10^4$, the BDF2 method performs better in short time but has bigger errors than BE in long time. This observation is expected because dissipation controls short-time errors, and convection controls long-time errors. We expect the BDF2 scheme to be better than BE scheme in short time, but BE being better than BDF2 in long time. Also, in the referenced test, the errors for without nudging grow to $\mathcal{O}(1)$ over time. In our test, the errors for without nudging do not grow at that rate. This may be due to the difference in initial conditions. The referenced test began the simulation with non-zero approximate velocity.

\subsection{Experiment 3: Pipe cavity flow problem}
This experiment is adapted from Ervin, Layton, and Maubach \cite{Layton:2000}. We test the performance of modular nudging on a more interesting, complex domain with high Reynolds number. We make comparisons with standard and without nudging. The domain, depicted in Figure 3, is a horizontal pipe of length 8 and height 1 with a rectangular cavity on top of the pipe of length 6 and height 5. Formally, the domain is 
\begin{align}
    \Omega = ([0,8]\times [0,1])\cup ([1,7]\times [1,6]). \notag
\end{align}
We assume no-slip boundary conditions on the walls and no force $f(x,y,t)=0$. We have inflow of fluid from the left end of the pipe and flows out at the right end, both with velocity $u(x,y)=\min(t,1)(4y(1-y))^T$. Because we are merely testing for accuracy, we choose simple inflow and outflow conditions. More complex inflow and outflow conditions may be needed in order to match experimental data \cite{Layton:2000}.

The solution is computed on a mesh with 10,433 dof. We set the final time $T=10$, time step size $\Delta t=\frac{1}{100}, \nu=10^{-3}, L=1, U=1$, and $Re=\frac{LU}{\nu}$. The initial velocity conditions are $u(x,y,0)=0$ and $v_0=I_H(u_0)$. We assume $I_H$ is $L^2$ projection and set $\chi=1$ and $10^4$.

We use Taylor-Hood $(P2-P1)$ finite element pair for the velocity and pressure spaces.
The nonlinear convection term is treated using a skew-symmetric formulation with divergence correction that takes the form
\begin{align}
    (v^{n+1}\cdot \nabla \Tilde{v}^{n+2}, w)+\frac{1}{2}((\nabla \cdot v^{n+1})\Tilde{v}^{n+2}, w), \text{ $\forall w \in X^H$}. \notag
\end{align}

The true solution $u$ is approximated using DNS, computed on a finer mesh with 41,194 dof. The DNS solution uses the BDF2 time discretization scheme, and the nonlinear term is treated exactly as above using the skew-symmetric formulation with divergence correction.

\begin{figure}[H]
    \centering
    \includegraphics[width=0.8\linewidth]{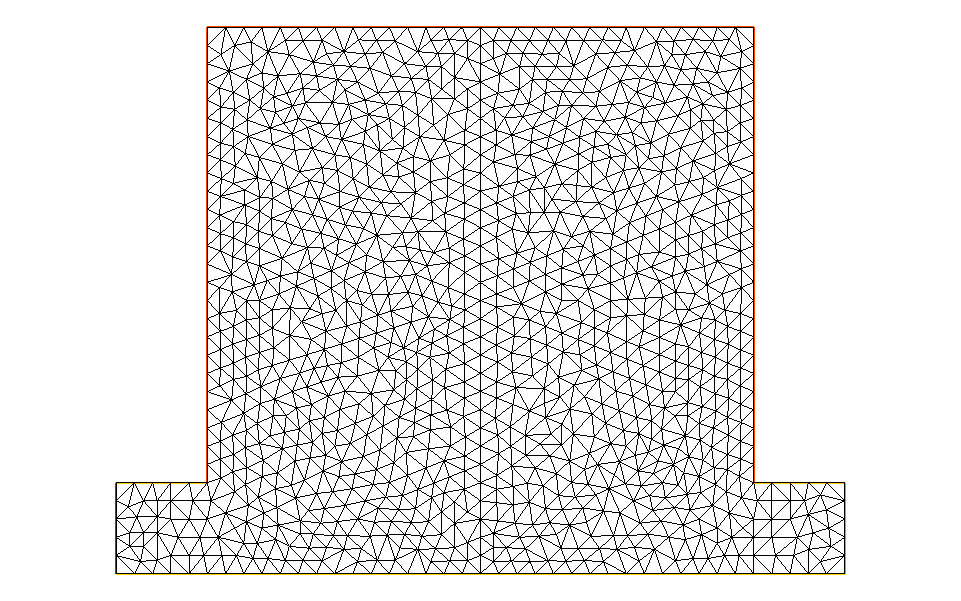}
    \caption{Pipe cavity domain.}
    \label{fig:Test3domain}
\end{figure}

Figures are plotted starting the second step to exclude the initial error being undefined on log scale. The numerical results are as follows.

\begin{figure}[H]
    \centering
    \includegraphics[width=0.8\linewidth]{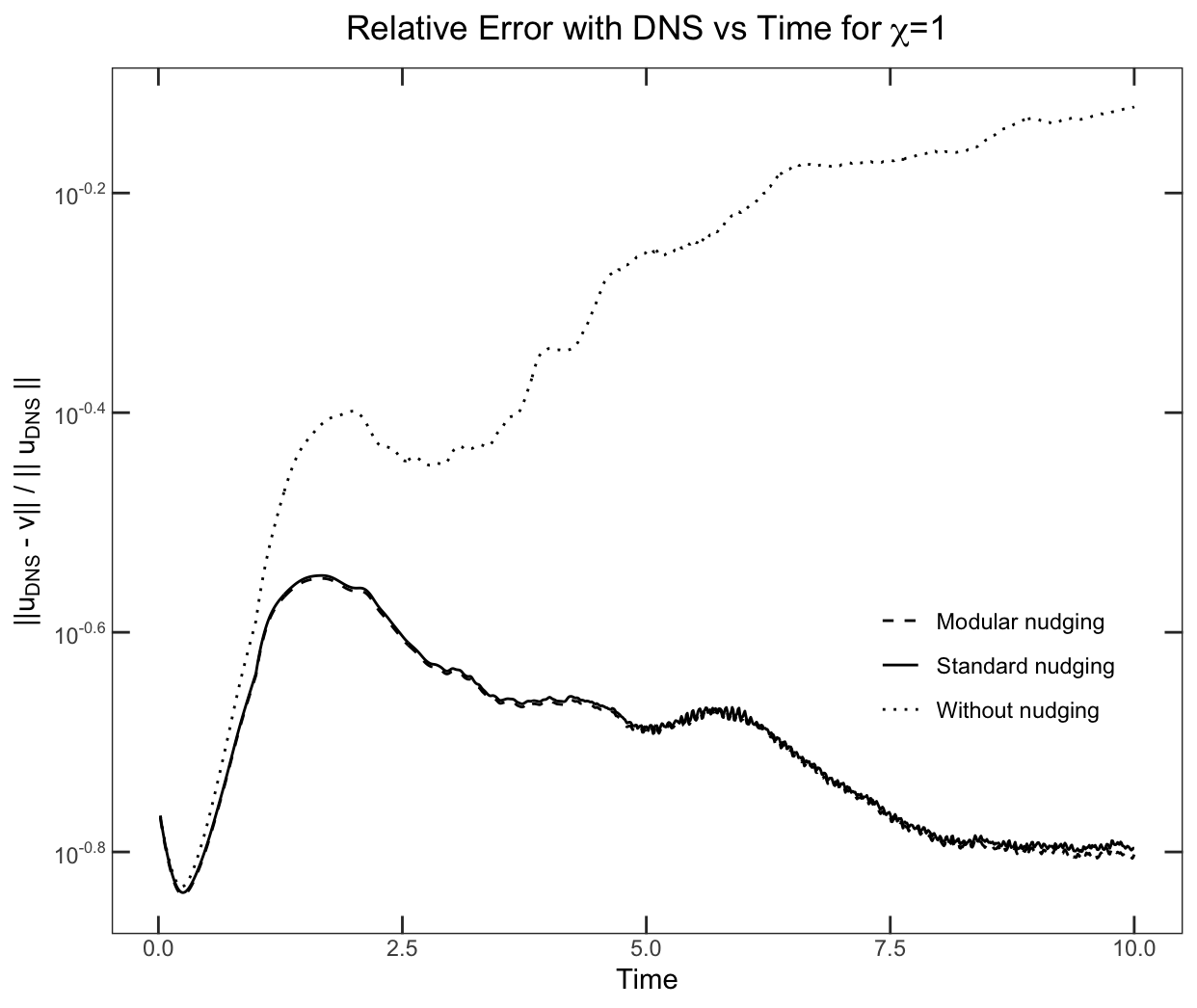}
    \caption{For $\chi=1$, modular nudging shares similar performance as standard nudging, both performing much better than without nudging.}
    \label{fig:Test3X=1}
\end{figure}

For $\chi=1$, the relative errors for the modular and standard nudging grow to $\mathcal{O}(10^{-0.6})$ at $t=2$ and decrease afterwards. The errors for both methods saturate around $\mathcal{O}(10^{-0.8})$. The relative error without nudging steadily increases to $\mathcal{O}(1)$. Modular and standard nudging have similar accuracy to each other. Without nudging performs much worse.

\begin{figure}[H]
    \centering
    \includegraphics[width=0.8\linewidth]{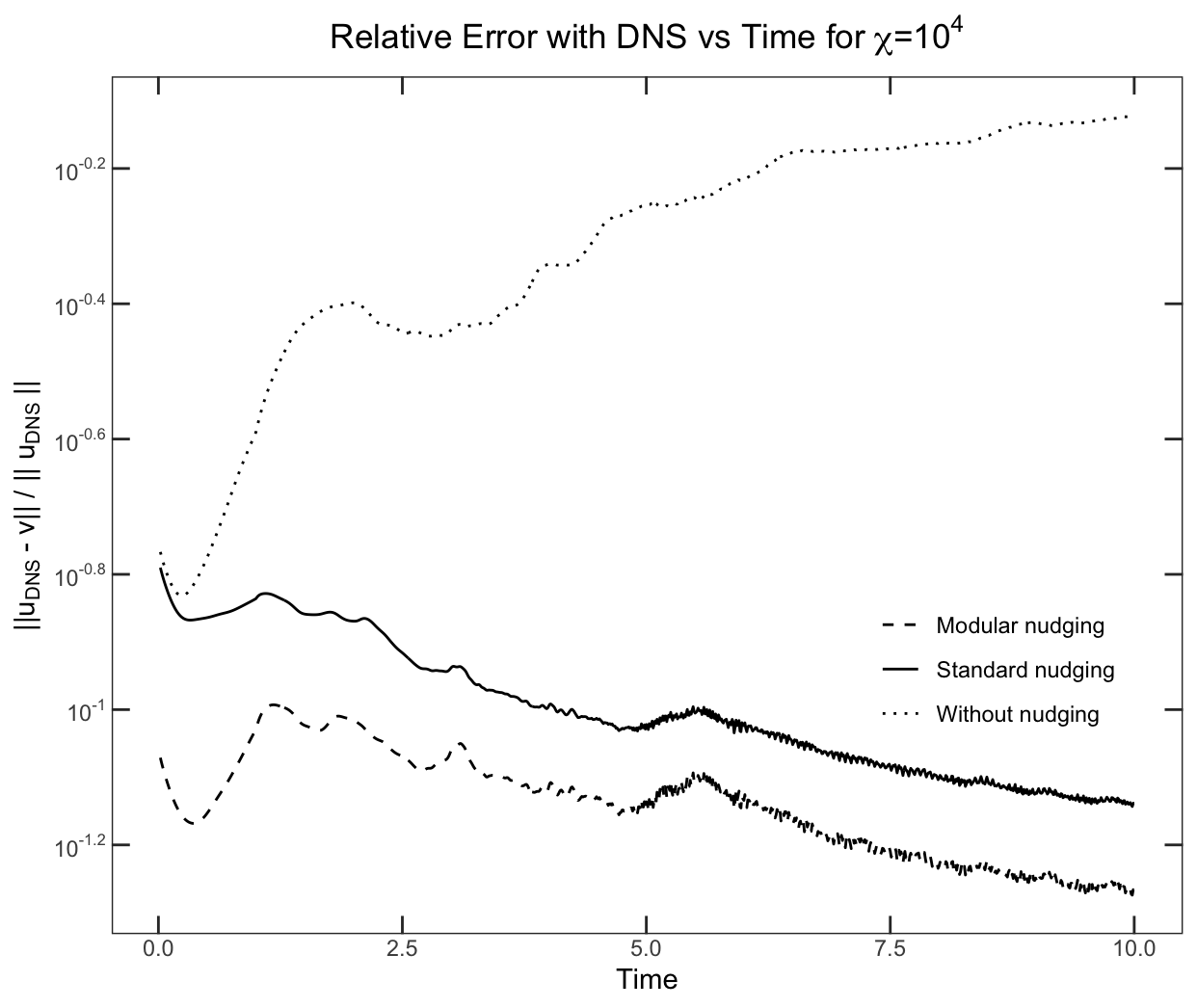}
    \caption{For $\chi=10^4$, modular nudging performs better than standard nudging and much better than without nudging.}
    \label{fig:Test2X=1e4}
\end{figure}

For $\chi=10^4$, the relative errors for the modular and standard nudging quickly decrease after after initial calculations. The errors for both methods saturate around $\mathcal{O}(10^{-1})$. The relative error without nudging steadily increases to $\mathcal{O}(1)$. In longtime, modular nudging performs better than standard nudging with no nudging having much larger errors.

\subsubsection{Test with do-nothing outflow}
In this test, we alter the outflow condition to "do-nothing" as proposed by Volker John on page 21 of \cite{John:2016}. We keep the same inflow condition and do not impose any condition on the outflow. We investigate the effect of different outflow condition on the error. The solution without nudging blows up in this test, so relative error is computed relative to the coarse solution for better view of finer scale comparisons.

\begin{figure}[H]
    \centering
    \includegraphics[width=0.8\linewidth]{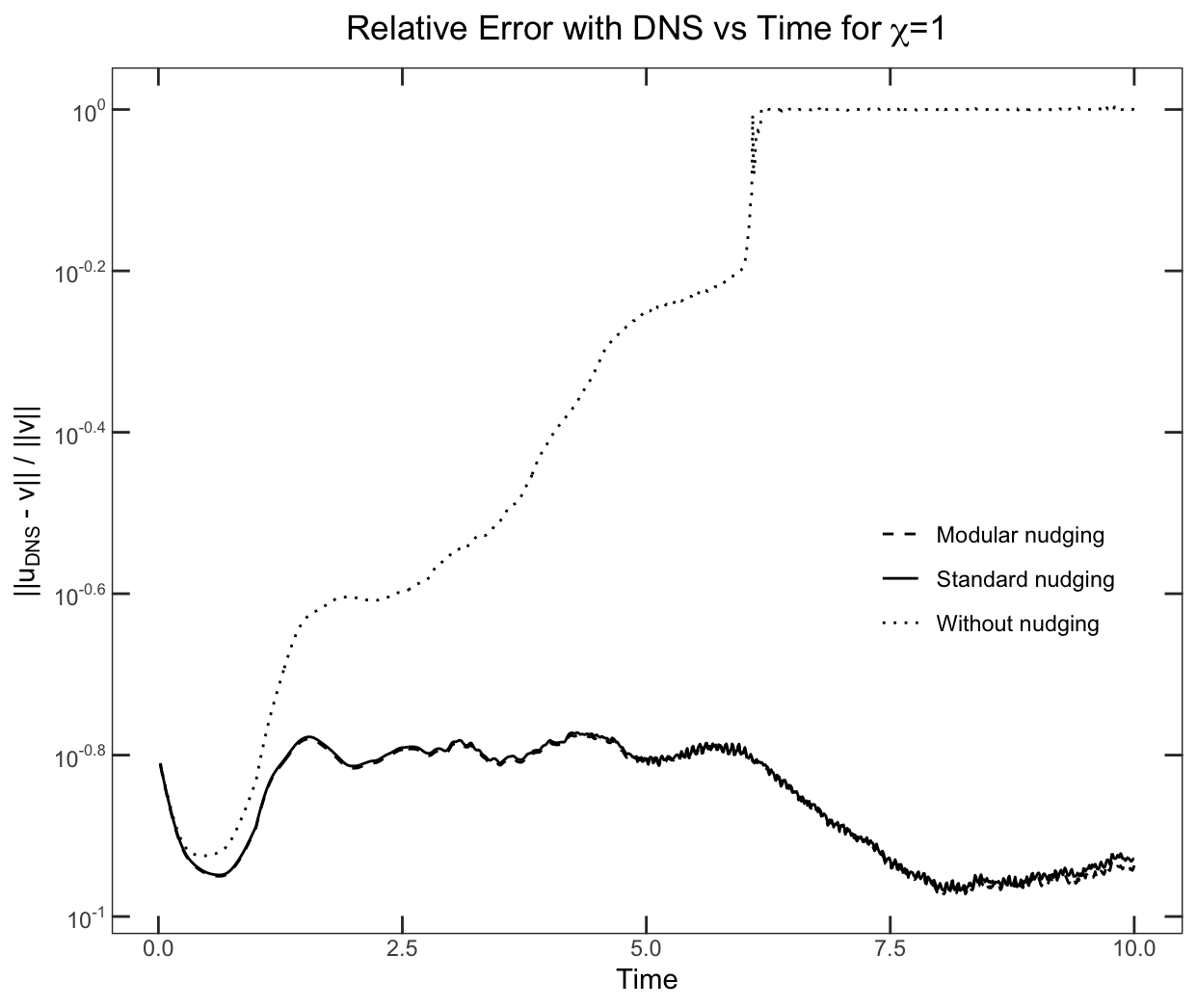}
    \caption{For $\chi=1$, modular nudging performs similar to standard nudging with both performing better than without nudging.}
    \label{fig:Test2X=1DN}
\end{figure}

For $\chi=1$, the relative errors for the modular and standard nudging hover around $\mathcal{O}(10^{-0.8})$ and saturate at $\mathcal{O}(10^{-1})$. The relative error for without nudging grows to $\mathcal{O}(1)$. Modular and standard nudging has comparable performance, both being better than without nudging.

\begin{figure}[H]
    \centering
    \includegraphics[width=0.8\linewidth]{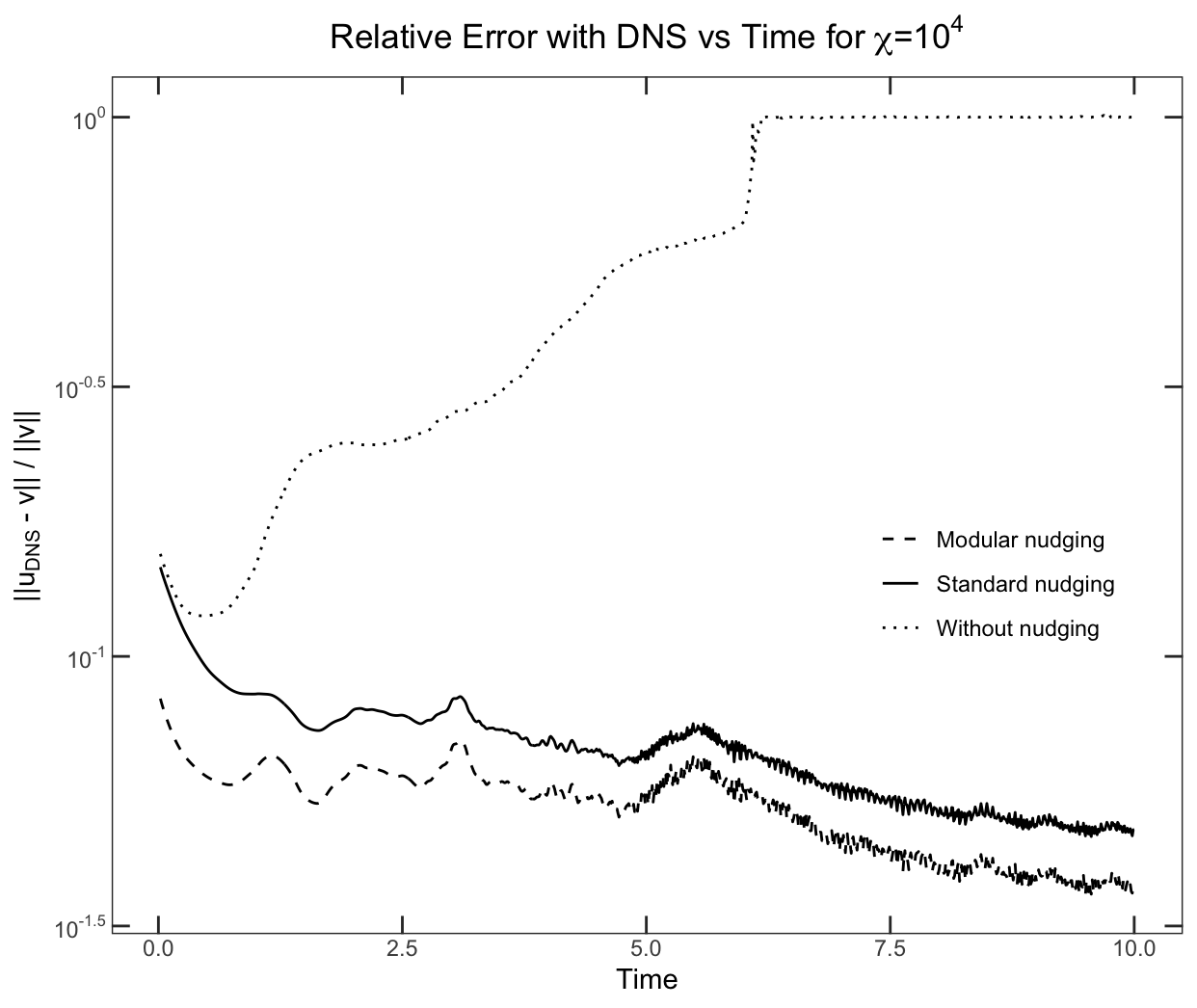}
    \caption{For $\chi=10^4$, modular nudging performs better than standard nudging and much better than without nudging.}
    \label{fig:Test2X=1e4DN}
\end{figure}

For $\chi=10^4$, the relative errors for the modular and standard nudging decrease steadily and saturate at around $\mathcal{O}(10^{-1.5})$. The relative error for without nudging quickly increases to $\mathcal{O}(1)$. In longtime, modular nudging has smaller errors than standard nudging, and the errors grow for without nudging.

In comparison to the errors from imposing the same conditions for inflow and outflow, the errors with the "do-nothing" condition are smaller. This observation holds for both $\chi=1$ and $10^4$. Although the error difference is minor, we do observe smaller error when imposing the "do-nothing" outflow condition.
  
\section{Conclusion and Future Prospects}
  This report analyzed and tested the BDF2 modular, 2-step nudging algorithm. If $I_H=I_H^2$, the analysis step can be explicit, while if $I_H\neq I_H^2$ explicit use of the derived formula adds a consistency error term. The numerical tests confirm the effectiveness of the method. 
We proved stability and long-time accuracy of the approximate solution. Numerical experiments support the observation that modular nudging is computationally cheaper than standard nudging while providing comparable performance. Future directions  include development of modular nudging with other time discretization schemes, resolving the complications with the alternate analysis step, and self-adapting the nudging parameter $\chi$.

\section{Appendix}
    \appendix
\section{Discussion about an alternate analysis step}
This section analyzes an alternate analysis step and discusses the bottlenecks encountered when attempting to prove stability and error estimates for it. The alternate analysis step takes the form:
\begin{align}
    \text{Forecast step: } &\frac{3\Tilde{v}^{n+2}-4v^{n+1}+v^n}{2\Delta t}+\tilde{v}^{n+2} \cdot \nabla \Tilde{v}^{n+2} - \nu \Delta \Tilde{v}^{n+2} + \nabla q^{n+2}=f(x) \\ &\nabla \cdot \Tilde{v}^{n+2} = 0 \\
    \text{Analysis step: }
    &\frac{3v^{n+2}-3\Tilde{v}^{n+2}}{2\Delta t}-\chi I_H(u(t^{n+2})-v^{n+2}) \\
    & - \nu \Delta (v^{n+2}-\Tilde{v}^{n+2}) - \nabla \cdot (\nu_{turb}(\Tilde{v}^{n+2})\nabla v^{n+2})=0. \notag
\end{align}

We now analyze the stability of this method. For the alternate analysis step, we have the following algorithm
    \begin{align}
    \text{Forecast step: } 
    &\text{Given } v^n, v^{n+1} \in X_h, \text{ find } (\Tilde{v}^{n+2}, q^{n+2}) \in (X_h, Q_h) \text{ satisfying:} \notag \\ 
    &\Bigg(\frac{3\Tilde{v}^{n+2}-4v^{n+1}+v^n}{2\Delta t}, v_h\Bigg)+b(\Tilde{v}^{n+2}, \Tilde{v}^{n+2}, v_h) + \nu (\nabla \Tilde{v}^{n+2}, \nabla v_h) \\
    &- (q^{n+2}, \nabla \cdot v_h)=(f^{n+2},v_h) \text{ } \forall v_h \in X_h, \notag \\ 
    &(\nabla \cdot \Tilde{v}^{n+2}, q_h) = 0 \text{ } \forall q_h \in Q_h.\\
    \text{Analysis step: }
    &\Bigg(\frac{3v^{n+2}-3\Tilde{v}^{n+2}}{2\Delta t}, v_h\Bigg) -\chi (I_H(u(t^{n+2})-v^{n+2}), v_h) \notag \\
    &+ \nu (\nabla (v^{n+2}-\Tilde{v}^{n+2}),\nabla v_h) - (\nabla \cdot (\nu_{turb}(\Tilde{v}^{n+2})\nabla v^{n+2}), v_h)=0 \text{ } \forall v_h \in X_h.
    \end{align}

As proceeded in stability analysis of the analysis step in section 4, we set $v_h=\Tilde{v}^{n+2}$ in (A.4) and $q_h=q^{n+2}$ in (A.5). Adding these two equations gives
\begin{align}
    &\Bigg(\frac{3\Tilde{v}^{n+2}-4v^{n+1}+v^n}{2\Delta t}, \Tilde{v}^{n+2}\Bigg) + \nu ||\nabla \Tilde{v}^{n+2}||^2 =(f^{n+2},\Tilde{v}^{n+2}).
\end{align}
We rewrite the first term as follows
\begin{align}
    \Bigg(\frac{3\Tilde{v}^{n+2}-4v^{n+1}+v^n}{2\Delta t}, \Tilde{v}^{n+2}\Bigg)=
    &\Bigg(\frac{3v^{n+2}-4v^{n+1}+v^n}{2\Delta t}, v^{n+2}\Bigg) + \Bigg(\frac{3\Tilde{v}^{n+2}-3v^{n+2}}{2\Delta t}, \Tilde{v}^{n+2}\Bigg) \notag \\
    &+\Bigg(\frac{3\Tilde{v}^{n+2}-4v^{n+1}+v^n}{2\Delta t}, \Tilde{v}^{n+2}-v^{n+2}\Bigg).
\end{align}
For the first term on the RHS in (A.8), we use Lemma 2.1. We apply the polarization identity to the second term. The third term is a problematic term. We rely on an identity deriving from the alternate analysis step.
If we rewrite the alternate analysis step as
\begin{align}
    &\Bigg(\frac{1}{2\Delta t}I-\nu \Delta\Bigg)(v^{n+2}-\Tilde{v}^{n+2})-\chi I_H(u^{n+2}-v^{n+2})=0 \\
    \implies& v^{n+2}-\tilde{v}^{n+2} = \Bigg(\frac{1}{2\Delta t}I-\nu \Delta\Bigg)^{-1}\chi I_H(u^{n+2}-v^{n+2}) \\
    \implies& \tilde{v}^{n+2}-v^{n+2} = \Bigg(\frac{1}{2\Delta t}I-\nu \Delta\Bigg)^{-1}\chi I_H(v^{n+2}-u^{n+2}).
\end{align}
We treat $\Psi:=(\frac{1}{2\Delta t}I-\nu \Delta)^{-1}$ as an operator. Then taking the inner product of (A.11) with $\frac{1}{2\Delta t}(3v^{n+2}-4v^{n+1}+v^n)$, we get
\begin{align}
    &\frac{1}{2\Delta t}(3v^{n+2}-4v^{n+1}+v^n, \tilde{v}^{n+2}-v^{n+2}) \notag \\ 
    =&\frac{1}{2\Delta t} (3v^{n+2}-4v^{n+1}+v^n, \Psi \chi I_H(v^{n+2}-u^{n+2})).
\end{align}
At this point, we can split $\Psi$ into $\Psi^{1/2}\Psi^{1/2}$ since $\Psi$ is symmetric positive definite. Then we have
\begin{align}
    &\frac{\chi}{2\Delta t}(3v^{n+2}-4v^{n+1}+v^n, \Psi I_H(v^{n+2}-u^{n+2})) \notag \\
    =&\frac{\chi}{2\Delta t}\Bigg(\Psi^{1/2}(3v^{n+2}-4v^{n+1}+v^n), \Psi^{1/2} I_H(v^{n+2}-u^{n+2})\Bigg).
\end{align}
Now if the operator $\Psi$ and the $L^2$ projection $I_H$ commute, then we apply $I_H$ to the left function on the RHS of (A.13) and get
\begin{align}
    \frac{\chi}{2\Delta t}\Bigg(\Psi^{1/2}I_H(3v^{n+2}-4v^{n+1}+v^n), \Psi^{1/2} I_H(v^{n+2}-u^{n+2})\Bigg).
\end{align}
From here, we can add and subtract $u^{n+2},u^{n+1},u^n$ to the left function of (A.13) and obtain
\begin{align}
    &\frac{\chi}{2\Delta t}\Bigg(\Psi^{1/2}I_H(3(v^{n+2}-u^{n+2})-4(v^{n+1}-u^{n+1})+(v^n-u^n)), \Psi^{1/2} I_H(v^{n+2}-u^{n+2})\Bigg) \notag \\
    +&\Bigg(\Psi^{1/2}I_H(3u^{n+2}-4u^{n+1}+u^n),\Psi^{1/2}I_H(v^{n+2}-u^{n+2})\Bigg).
\end{align}
For the first term, we apply Lemma 2.1. For the second term, we move it to the RHS and bound it using Cauchy-Schwarz inequality and Young's inequality. The remaining stability analysis is exactly the same as in section 4. Therefore the difficulty lies in proving the commutativity of the operators $\Psi$ and $I_H$. The error analysis encounters this exact problem.

\newpage

\bibliographystyle{IEEEtran}

\bibliography{mybib}

\end{document}